\documentclass[11pt,reqno]{amsart}

\usepackage{amsmath,amsthm,amssymb,comment,fullpage}
\usepackage{braket}
\usepackage{mathtools}

\usepackage{caption}
\usepackage{times}
\usepackage[T1]{fontenc}
\usepackage{mathrsfs}
\usepackage{latexsym}
\usepackage[dvips]{graphics}
\usepackage{epsfig}
\usepackage{amsmath,amsfonts,amsthm,amssymb,amscd}
\input amssym.def
\input amssym.tex
\usepackage{color}
\usepackage{hyperref}
\usepackage{url}
\newcommand{\bburl}[1]{\textcolor{blue}{\url{#1}}}

\usepackage{tikz}
\usepackage{tkz-tab}
\usetikzlibrary{shapes.geometric,positioning}

\newcommand{\burl}[1]{\textcolor{blue}{\url{#1}}}

\numberwithin{equation}{section}

\newtheorem{thm}{Theorem}[section]

\newtheorem{cor}[thm]{Corollary}
\newtheorem{lem}[thm]{Lemma}

\theoremstyle{plain}

\newtheorem{corollary}[thm]{Corollary}
\newtheorem{definition}[thm]{Definition}

\newtheorem{lemma}[thm]{Lemma}
\newtheorem{proposition}[thm]{Proposition}
\newtheorem{theorem}[thm]{Theorem}

\newcommand\be{\begin{equation}}
\newcommand\ee{\end{equation}}
\newcommand\bee{\begin{equation*}}
\newcommand\eee{\end{equation*}}
\newcommand\bea{\begin{eqnarray}}
\newcommand\eea{\end{eqnarray}}
\newcommand\beae{\begin{eqnarray*}}
\newcommand\eeae{\end{eqnarray*}}
\newcommand\bi{\begin{itemize}}
\newcommand\ei{\end{itemize}}
\newcommand\ben{\begin{enumerate}}
\newcommand\een{\end{enumerate}}
\newcommand\bc{\begin{center}}
\newcommand\ec{\end{center}}
\newcommand\ba{\begin{array}}
\newcommand\ea{\end{array}}

\newcommand{\pn}[1]{\left( #1 \right)}

\newcommand{\Prob}[1]{\ensuremath{{\mathbb P}\left( #1 \right)}}

\newcommand\frakfamily{\usefont{U}{yfrak}{m}{n}}
\DeclareTextFontCommand{\textfrak}{\frakfamily}

\DeclareMathOperator{\Var}{Var}

\newtheorem{rek}[thm]{Remark}

\newcommand{\hr}[1]{\href{#1}{\url{#1}}}

\makeatletter
\def\@tocline#1#2#3#4#5#6#7{\relax
  \ifnum #1>\c@tocdepth 
  \else
    \par \addpenalty\@secpenalty\addvspace{#2}%
    \begingroup \hyphenpenalty\@M
    \@ifempty{#4}{%
      \@tempdima\csname r@tocindent\number#1\endcsname\relax
    }{%
      \@tempdima#4\relax
    }%
    \parindent\z@ \leftskip#3\relax \advance\leftskip\@tempdima\relax
    \rightskip\@pnumwidth plus4em \parfillskip-\@pnumwidth
    #5\leavevmode\hskip-\@tempdima
      \ifcase #1
       \or\or \hskip 1em \or \hskip 2em \else \hskip 3em \fi%
      #6\nobreak\relax
    \hfill\hbox to\@pnumwidth{\@tocpagenum{#7}}\par
    \nobreak
    \endgroup
  \fi}
\makeatother

\title{Generalizing the Distribution of Missing Sums in Sumsets}

\author{H\`{u}ng Vi\d{\^{e}}t Chu}
\email{\textcolor{blue}{\href{mailto:hungchu2@illinois.edu}{hungchu2@illinois.edu}}}
\address{Department of Mathematics, University of Illinois at Urbana Champaign, Urbana, IL 61820}

\author{Dylan King}
\email{\textcolor{blue}{\href{mailto:kingda16@wfu.edu}{kingda16@wfu.edu}}}
\address{Department of Mathematics, Wake Forest University, Winston-Salem, NC 27109}

\author{Noah Luntzlara}
\email{\textcolor{blue}{\href{mailto:nluntzla@umich.edu}{nluntzla@umich.edu}}}
\address{Department of Mathematics, University of Michigan, Ann Arbor, MI 48109}

\author{Thomas C. Martinez}
\email{\textcolor{blue}{\href{mailto:tmartinez@hmc.edu}{tmartinez@hmc.edu}}}
\address{Department of Mathematics, Harvey Mudd College, Claremont, CA 91711}

\author{Steven J. Miller}
\email{\textcolor{blue}{\href{mailto:sjm1@williams.edu}{sjm1@williams.edu}},  \textcolor{blue}{\href{Steven.Miller.MC.96@aya.yale.edu}{Steven.Miller.MC.96@aya.yale.edu}}}
\address{Department of Mathematics and Statistics, Williams College, Williamstown, MA 01267}

\author{Lily Shao}
\email{\textcolor{blue}{\href{mailto:ls12@williams.edu}{ls12@williams.edu}}}
\address{Department of Mathematics and Statistics, Williams College, Williamstown, MA 01267}

\author{Chenyang Sun}
\email{\textcolor{blue}{\href{mailto:cs19@williams.edu}{cs19@williams.edu}}}
\address{Department of Mathematics and Statistics, Williams College, Williamstown, MA 0126}

\author{Victor Xu}
\email{\textcolor{blue}{\href{mailto:vzx@andrew.cmu.edu}{vzx@andrew.cmu.edu}}}
\address{Department of Mathematics, Carnegie Mellon University, Pittsburgh, PA 15289}

\thanks{This work was partially supported by NSF grants DMS1659037 and DMS1561945.}


\date{\today}

\begin{document}

\maketitle

\begin{abstract} Given a finite set of integers $A$, its sumset is $A+A:= \{a_i+a_j \mid a_i,a_j\in A\}$. We examine $|A+A|$ as a random variable, where $A\subset I_n = [0,n-1]$, the set of integers from 0 to $n-1$, so that each element of $I_n$ is in $A$ with a fixed probability $p \in (0,1)$. Recently, Martin and O'Bryant studied the case in which $p=1/2$ and found a closed form for $\mathbb{E}[|A+A|]$. Lazarev, Miller, and O'Bryant extended the result to find a numerical estimate for $\text{Var}(|A+A|)$ and bounds on the number of missing sums in $A+A$, $m_{n\,;\,p}(k) := \mathbb{P}(2n-1-|A+A|=k)$. Their primary tool was a graph-theoretic framework which we now generalize to provide a closed form for $\mathbb{E}[|A+A|]$ and $\text{Var}(|A+A|)$ for all $p\in (0,1)$ and establish good bounds for $\mathbb{E}[|A+A|]$ and $m_{n\,;\,p}(k)$.\\
\indent We continue to investigate $m_{n\,;\,p}(k)$ by studying $m_p(k) = \lim_{n\to\infty}m_{n\,;\,p}(k)$, proven to exist by Zhao. Lazarev, Miller, and O'Bryant proved that, for $p=1/2$, $m_{1/2}(6)>m_{1/2}(7)<m_{1/2}(8)$. This distribution is not unimodal, and is said to have a ``divot'' at 7. We report results investigating this divot as $p$ varies, and through both theoretical and numerical analysis, prove that for $p\geq 0.68$ there is a divot at $1$; that is, $m_{p}(0)>m_{p}(1)<m_{p}(2)$.\\
\indent Finally, we extend the graph-theoretic framework originally introduced by Lazarev, Miller, and O'Bryant to correlated sumsets $A+B$ where $B$ is correlated to $A$ by the probabilities $\mathbb{P}(i\in B \mid i\in A) = p_1$ and $\mathbb{P}(i\in B \mid i\not\in A) = p_2$. We provide some preliminary results using the extension of this framework.
\end{abstract}

\tableofcontents


\section{Introduction}\label{intro}

Many problems in additive number theory, such as Fermat's Last Theorem or
the Goldbach conjecture or the infinitude of twin primes, can be cast as problems involving sum or
difference sets. For example, if $P_n$ is the set of
$n$\textsuperscript{th} powers of positive integers, Fermat's Last Theorem
is equivalent to $(P_n + P_n) \cap P_n = \varnothing$ for $n\ge 3$. Given a finite set of
non-negative integers $A$, we define
the sumset $A+A := \{a_i+a_j \mid a_i,a_j\in A\}$ and the difference set $A-A
:= \{a_i-a_j \mid a_i,a_j\in A\}$. The set $A$ is said to be
\begin{itemize}
\item {\it sum-dominant} if $|A+A|>|A-A|$ (also called {\it MSTD}, or {\it More Sums Than Differences}),
\item {\it balanced} if $|A+A| = |A-A|$, and
\item {\it difference-dominant} if $|A+A|<|A-A|$.
\end{itemize}
By $[a,b]$, we mean the set of integers $\{a, a+1, \dots, b\}$. As addition is commutative and subtraction is not, it was expected that in the limit almost all sets would be difference-dominant, though there were many constructions of infinite families of MSTD sets.\footnote{The proportion of sets in $[0,n-1]$ in these families tend to zero as $n \to \infty$. In the early constructions these densities tended to zero exponentially fast, but recent methods have found significantly larger ones where the decay is polynomial.} There is an extensive literature on such sets, their constructions, and generalizations to settings other than subsets of the integers; see for example \cite{AMMS, BELM, CLMS, CMMXZ, DKMMW, He, HLM, ILMZ, Ma, MOS, MS, MPR, MV, Na1, Na2, PW, Ru1, Ru2, Ru3, Sp, Zh1}.

We are interested in studying $|A+A|$ as we randomly choose $A$ using a Bernoulli process. Explicitly, we fix a $p \in (0,1)$ and construct $A \subset [0,n-1]$ by independently including each $i \in [0, n-1]$ to be in $A$ with probability $p$. Martin and O'Bryant \cite{MO} studied the distributions of $|A+A|$ and $|A-A|$ when $p=1/2$, including computing the expected values.  Contrary to intuitions, they proved a positive percentage of these sets are MSTD in the limit as $n\to \infty$. Note that $p = 1/2$ is equivalent to the model where each subset of $[0, n-1]$ is equally likely to be chosen. Their work extends to any fixed $p > 0$, though if $p$ is allowed to decay to zero with $n$ then the intuition is correct and almost all sets are difference dominated \cite{HM}.

Lazarev, Miller and O'Bryant \cite{LMO} continued this program in the special but important case of $p=1/2$. They computed the variance of $|A+A|$, showed that the distribution is asymptotically exponential, and proved the existence of a ``divot'', which we now explain. From \cite{MO}, the expected number of missing sums is $10$ as $n \to \infty$; thus almost all sets are missing few sums, making it more convenient to plot the distribution of the number of missing sums. For $A \subset [0, n-1]$, we set $m_{n\,;\,p}(k) :=\mathbb{P}(2n-1 - |A+A| = k)$, and examine the distribution of $m_p(k):=\lim_{n \to \infty} m_{n\,;\,p}(k)$, proven to exist by Zhao \cite{Zh2}. The distribution does not just rise and fall, but forms a `divot', with $m_{1/2}(6) > m_{1/2}(7) < m_{1/2}(8)$; see Figure \ref{fig:LMO} for data and \cite{LMO} for a proof.

\begin{figure}[ht]
\begin{center}
\includegraphics[scale=.260]{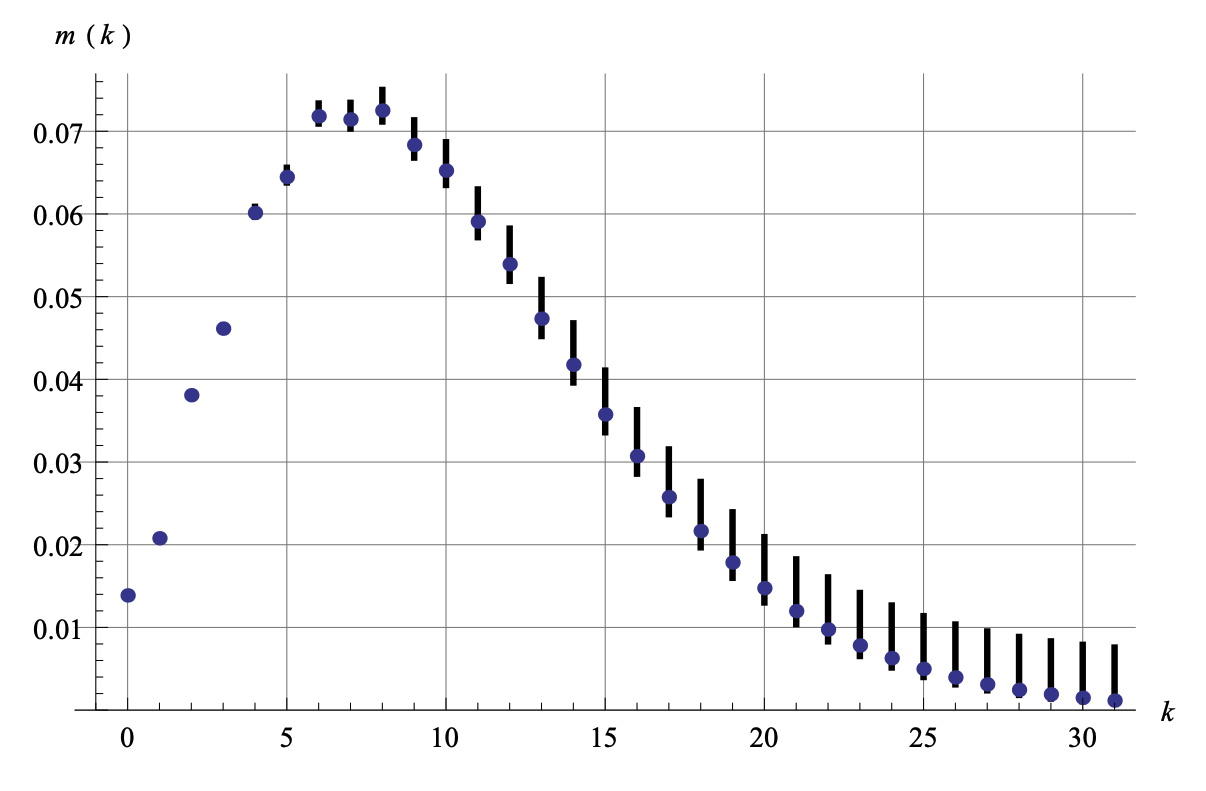}
\end{center}
\caption{From \cite{LMO}: Experimental values of $m_{p}(k)$, where $p=1/2$, with vertical bars depicting the values allowed by our rigorous bounds. In most cases, the allowed interval is smaller than the dot indicating the experimental value. The data comes from generating $2^{28}$ sets uniformly forced to contain 0 from $[0, 256)$.\label{fig:LMO}}
\end{figure}

We extend the methodologies developed in \cite{LMO} to study the distribution of $|A+A|$ for generic $p$ not necessarily equal to $1/2$; there are many technical issues that arise which greatly complicate the combinatorial analysis when $p \neq 1/2$. To do so, we generalize many previous results in Section \ref{sec:generalMO}, and use them to derive a formula for the expected value of $|A+A|$, which we then analyze.

\begin{theorem}\label{thm:expectedvalue}
Let $A\subseteq [0,n-1]$ with $\mathbb{P}(i\in A) = p$ for $p\in(0,1)$. Then $\mathbb{E}[|A+A|]$ equals
\begin{equation}\label{eqn:expectedvalueformulaINTHM}
    \sum_{r=0}^{n}\,\binom{n}{r}\ p^r\, q^{n-r}\left(2\sum_{i=0}^{n-2}(1-\mathbb{P}(i \not \in A+A \ | \ |A|=r))+(1-\mathbb{P}(n-1 \not \in A+A \ | \ |A|=r))\right),
\end{equation}
where $q=1-p$ and
\begin{equation}\label{eqn:p_k_notin_s_condINTHM}
    \mathbb{P}(i \not \in A+A \ | \ |A|=r)= \begin{cases}\frac{\displaystyle\sum_{k\, =\, 0}^{\frac{i+1}2}2^k\binom{\frac{i+1}{2}}{k}\binom{n-i-1}{r-k}}{\displaystyle\binom{n}{r}} &\text{ {\rm for}\ i \text{\rm odd},}\\
    \frac{\displaystyle\sum_{k\, =\, 0}^{\frac{i}2}2^k\binom{\frac{i}{2}}{k}\binom{n-i-1}{r-k-1}}{\displaystyle\binom{n}{r}}&\text{ {\rm for}\ i \text{{\rm even.}}}
    \end{cases}
\end{equation}
\end{theorem}

As we need to compute on the order of $n^3$ sums to compute $\mathbb{E}[|A+A|]$, an useful bound is needed for numerical investigations.

\begin{theorem}\label{thm:expectedvaluebounds}
Let $A\subseteq [0,n-1]$ with $\mathbb{P}(i\in A) = p$ for $p\in (0,1)$ and set $q = 1-p$. Then
\be
    \mathbb{E}[|A+A|]\ \leq\ 2n-1-2q\ \frac{1-q^{\frac{n-1}2}}{1-\sqrt{q}}.
\ee
If $p > 1/2$, then we also get
\be
    \mathbb{E}[|A+A|]\ \geq\ 2n-1-2q\ \frac1{1-\sqrt{2q}} - \pn{2q}^{\frac{n-1}2}.
\ee
\end{theorem}

The proofs of Theorems \ref{thm:expectedvalue} and \ref{thm:expectedvaluebounds} are given in Section \ref{sec:expecvalue}. The proofs require an extension of the graph-theoretic framework of \cite{LMO}, which is done in Section \ref{sec:graph}.

We also compute the variance of $|A+A|$.

\begin{theorem}\label{thm:variance}
Let $A\subseteq [0,n-1]$ with $\mathbb{P}(i\in A)\, =\, p$ with $p\in (0,1)$. Then
\begin{equation}
\Var(|A+A|)=\sum_{r=0}^n\, \binom{n}{r}\,p^r\, q^{n-r}\Bigg(2\sum_{0\,\leq\, i\, < \,j\, \leq \,2n-2} 1-P_r(i,j)\, + \sum_{0\,\leq \,i\, \leq \,2n-2} 1-P_r(i)  \Bigg)\, -\, \mathbb{E}[|A+A|]^2,
\end{equation}
where $q=1-p$, $\mathbb{E}[|A+A|]$ is as calculated in Theorem ~\ref{thm:expectedvalue}, $P_r(i) = \mathbb{P}(i\not\in A+A \ | \ |A|=r)$ and $P_r(i,j) = \mathbb{P}(i\text{ and }j \not\in A+A \ | \ |A|=r)$.
\end{theorem}

As opposed to the calculation for the expected value, the variance does not have an easily calculable closed form, as $\mathbb{P}(i\text{ and }j \not\in A+A \ | \ |A|=r)$ has on the order of $p(n)$ terms to calculate, where $p(n)$ is the partition number of $n$ and grows faster than any polynomial\footnote{One has $\log p(n) \sim \pi\sqrt{3n/2}$.}. We discuss these issues in Section \ref{sec:variance}, where we prove Theorem \ref{thm:variance}. The difficulty arises as we go from $\mathbb{P}(i \not\in A+A \ | \ |A|=r)$ to $\mathbb{P}(i\text{ and }j \not\in A+A \ | \ |A|=r)$ because we introduce many more dependencies between nodes in the graph-theoretic framework. We were, however, able to show that the number of missing sums is asymptotically exponential.

\begin{theorem}\label{thm:m of k}
Let $A\subseteq [0,n-1]$ with $\mathbb{P}(i\in A)=p$ for $p \in (0,1)$ and recall that $m_{n\,;\,p}(k):=\mathbb{P}(2n-1 - |A+A| = k)$. If $n > 2 \frac{\log(1-p)}{\log(1-p^2)}k$, then
\[
    q^{k/2}\ \ll\ m_{n\,;\,p}(k)\ \ll\ \pn{\frac{1-p+g(p)}{2}}^k,
\]
where $g(p)\ =\ \sqrt{1+2p-3p^2}.$
\end{theorem}
The proof of Theorem ~\ref{thm:m of k} is structurally equivalent to the proof of Theorem 1.2 in \cite{LMO}. The full proof is in Appendix ~\ref{app:proofs}, but the idea is to study specific scenarios that are less likely or more likely to happen, for the lower bound and the upper bound, respectively, using many of the results proven in Section \ref{sec:generalMO}.

Our next result investigates the shape of the distribution of $m_p(k)$. Recall that Zhao proved that $m_p(k)$ existed by fringe analysis \cite{Zh2}. The technique of fringe analysis for estimating numbers of missing sums is the means by which most results about sum-distributions have been obtained and is the method we will follow. The technique grows out of the observation that sumsets usually have fully populated centers: there is a very low probability that there will be any element missing that is not near one of the ends. When a suitable distance from the edge is chosen, this observation can be made precise by bounding the probability of missing any elements in the middle. It follows that most of the time, all missing sums must be near the ends of the interval; the only contribution to these elements is from the upper and lower fringes of the randomly chosen set. Conveniently, as long as they are short relative to the length of the whole set, the fringes are independent and can be analyzed separately from the rest of the elements. As long as they are reasonably sized (fewer than 30 elements, usually) a computer can numerically check by brute force all the possible fringe arrangements, and give exact data for the number of missing sums near the edges.\\

Working with difference sets is \emph{orders of magnitude} more challenging than working with sumsets. This is because the fringe method fails, since there are interactions between the upper and lower fringes when we consider difference sets and thus the computational time is the square of that for sumsets. No suitable alternative technique has been found, and so rigorous numerical results about difference sets are scarce. Most work, including ours, focuses on distributions of sums, though see \cite{H-AMP} for some results on differences.\\

The results for $p = 1/2$ were possible because there were nice interpretations for the terms that simplified the analysis; we do not have that in general, which is why our results are concentrated on the larger values of $p$; see Figure \ref{plot:numappmk}. In Section \ref{sec:divot}, we look for divots other than that at $m_{1/2}(7)$, and our main theorem is the following.

\begin{figure}[ht!]
\centering
\includegraphics[scale=.9]{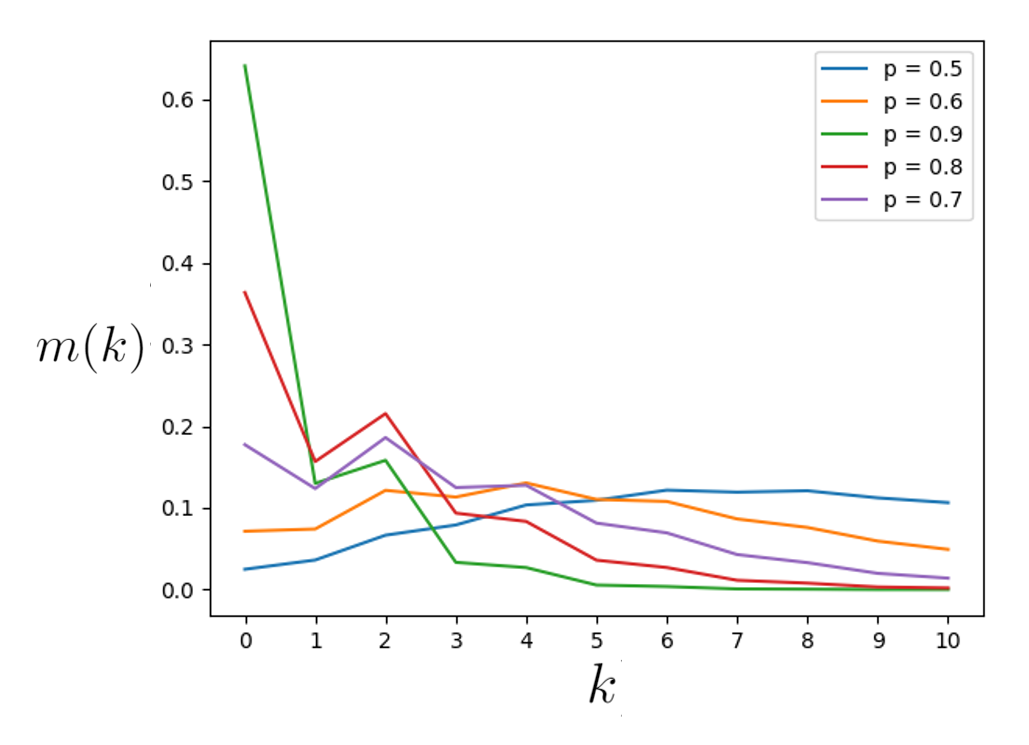}
\caption{\label{plot:numappmk} Plot of numerical approximations to $m_p(k)$, varying $p$ by simulating $10^6$ subsets of $\{0,1,2,\dots,400\}$. The simulation shows that: for $p=0.9$ and $0.8$, there is a divot at $1$, for $p=0.7$, there are divots at $1$ and $3$, for $p=0.6,$ there is a divot at $3$ and for $p=0.5$, there is a divot at $7$.}
\label{fig: intro}
\end{figure}

\begin{thm}\label{thm:divot}
For $p\ge 0.68$, there is a divot at $1$; that is, $m_p(0)>m_p(1)<m_p(2)$.
\end{thm}

Our final result looks at the generalization of our work to correlated sumsets (see \cite{DKMMW} for earlier work and results). We examine the random variable $|A+B|$, where, for a given triplet $(p,p_1,p_2)$ and any $i \in \{0,\dots,n-1\}$ we have

\begin{itemize}
    \item   $\mathbb{P}(i \in A) = p$,
    \item $\mathbb{P}(i \in B\ | \ i \in A) = p_1$, and
    \item $\mathbb{P}(i \in B\ | \ i \not \in A) = p_2$.
\end{itemize}

We extend our graph-theoretic framework to analyze this system and find $\mathbb{P}(k\not\in A+B)$ and $\mathbb{P}(i\text{ and } j \not\in A+B)$ in Section \ref{sec:correlated}. We end by considering some work that can be done in continuation of that presented here, in Section \ref{sec:future}.

\ \\

\section{Generalizations of \cite{MO}}\label{sec:generalMO}

We need to extend many of the lemmas and propositions from \cite{MO}, and prove they are true for general $p$ and not just $p=1/2$. The arguments typically do not change, so we only introduce notation as necessary\footnote{We just replace $1/2$ and $3/4$ with $q$ and $1-p^2$, respectively, as these are the representations of the exact values used in \cite{MO}.}; thus, we just state the results we use and how we generalized the argument. The full proofs are in Appendix ~\ref{app:proofs}.

\begin{lemma}[Lemma 5 of \cite{MO}]\label{lem:lemma5MO}
Let $n,\ell,u$ be integers with $n\geq \ell+u$. Fix $L\subset \{0,\dots,  \ell-1\}$ and $U\subset\{n-u,\dots,  n-1\}$. Suppose $R$ is a random subset of $\{\ell,\dots,n-u-1\}$, where each element of $\{\ell,\dots,n-u-1\}$ is in $R$ with independent probability $p\in(0,1)$, and define $A:=L\ \cup\ R\ \cup\ U$ and $q:=1-p$. Then for any integer $i$ satisfying $2\ell-1 \leq i \leq n-u-1$, we have
\begin{equation}
    \mathbb{P}(i\not \in A+A) \ = \ \begin{cases}q^{|L|}(1-p^2)^{\frac{i+1}{2}-\ell} & \mbox{{\rm if} } i \text{ {\rm odd,}} \\
    q^{|L|+1}(1-p^2)^{\frac{i}{2}-\ell} & \mbox{{\rm if} } i \text{ {\rm even}}.
    \end{cases}
\end{equation}
\end{lemma}

\begin{lemma}[Lemma 6 of \cite{MO}]\label{lem:lemma6MO}
Let $n,\ell,u$ be integers with $n\geq \ell+u$. Fix $L\subset \{0,\dots,  \ell-1\}$ and $U\subset\{n-u,\dots,  n-1\}$. Suppose $R$ is a random subset of $\{\ell,\dots,n-u-1\}$, where each element of $\{\ell,\dots,n-u-1\}$ is in $R$ with independent probability $p\in(0,1)$, and define $A:=L\ \cup\ R\ \cup\ U$ and $q:= 1-p$. Then for any integer $i$ satisfying $n+\ell-1 \leq i \leq 2n-2u-1$, we have
\begin{equation}
    \mathbb{P}(i\not \in A+A) \ = \ \begin{cases}q^{|U|}(1-p^2)^{n-\frac{i+1}{2}-u} & \mbox{{\rm if} } i \text{ {\rm odd,}} \\
    q^{|U|+1}(1-p^2)^{n-1-\frac{i}{2}-u} & \mbox{{\rm if} } i \text{ {\rm even}}.
    \end{cases}
\end{equation}
\end{lemma}

\begin{lem}\label{lem:lemma7MO}
Choose $A\subseteq [0,n-1]$ by including each element with probability $p$. Set $q=1-p$. Then, for $0\le i\le n-1$, the probability
\be\mathbb{P}(i\notin A+A)\ =\ \begin{cases}
(2q-q^2)^{(i+1)/2} &\mbox{{\rm if} } i \text{ {\rm odd,}}\\
q\,(2q-q^2)^{i/2} &\mbox{{\rm if} } i \text{ {\rm even,}}
\end{cases}\ee
while for any integer $n-1\le i\le 2n-2$ the probability
\be\mathbb{P}(i\notin A+A)\ =\ \begin{cases}(2q-q^2)^{n-(i+1)/2} & \mbox{{\rm if} } i \text{ {\rm odd,}}\\
q\,(2q-q^2)^{n-1-i/2} & \mbox{{\rm if} } i \text{ {\rm even}}.\end{cases}\ee
\end{lem}

These give us a generalization of Proposition 8 from \cite{MO}.
\begin{proposition}[Proposition 8 of \cite{MO}]\label{prop:prop8MO}
Let $n,\ell,u$ be integers with $n\geq \ell+u$. Fix $L\subset \{0,\dots,  \ell-1\}$ and $U\subset\{n-u,\dots,  n-1\}$. Suppose $R$ is a random subset of $\{\ell,\dots,n-u-1\}$, where each element of $\{\ell,\dots,n-u-1\}$ is in $R$ with independent probability $p\in(0,1)$, and define $A:=L \cup R \cup U$ and $q:= 1-p$. Then the probability that
\be
    \{2\ell-1, \dots,n-u-1\}\cup \{n+\ell-1,\dots,2n-2u-1\}\ \subseteq\ A+A
\ee
is greater than $1 - \frac{1+q}{p^2}\pn{q^{|L|}+q^{|U|}}$.
\end{proposition}

\section{Graph-Theoretic Framework}\label{sec:graph}



We develop a graph-theoretic framework which has proved powerful in computing various probabilities used in calculations. As we have shown in Section~\ref{sec:generalMO}, we have an explicit formula for $\mathbb{P}(i \not\in A+A)$ (Lemmas ~\ref{lem:lemma5MO} and ~\ref{lem:lemma6MO}). However, for generic $i$ and $j$, $\mathbb{P}(i \not\in A+A)$ and $\mathbb{P}(j \not\in A+A)$ are dependent, and therefore $\mathbb{P}(i \text{ and } j \not\in A+A)$ requires more work. To understand the dependencies between these two events, we create a \textit{condition graph}, as defined in \cite{LMO}, with some slight modifications. In \cite{LMO}, $V=[0,\max\{i,j\}]$, while we use $V=[0,n-1]$. This distinction is because in \cite{LMO} there was no need to consider the unconnected vertices, but here they will prove meaningful for computations.

\begin{definition}\label{def:conditiongraph}
For a set $F \subseteq  [0,2n-2]$ we define the \emph{condition graph} $G_F=(V,E)$ induced on $V=[0,n-1]$ by $F$ where two vertices $k_1$ and $k_2$ share an edge $(k_1,k_2) \in E$ if $k_1+k_2 \in F$. For notational convenience, if $F = \{i,j\}$, we denote $G_F$ by $G_{i,j}$.
\end{definition}

\begin{figure}[ht]

\includegraphics[width=0.6\textwidth]{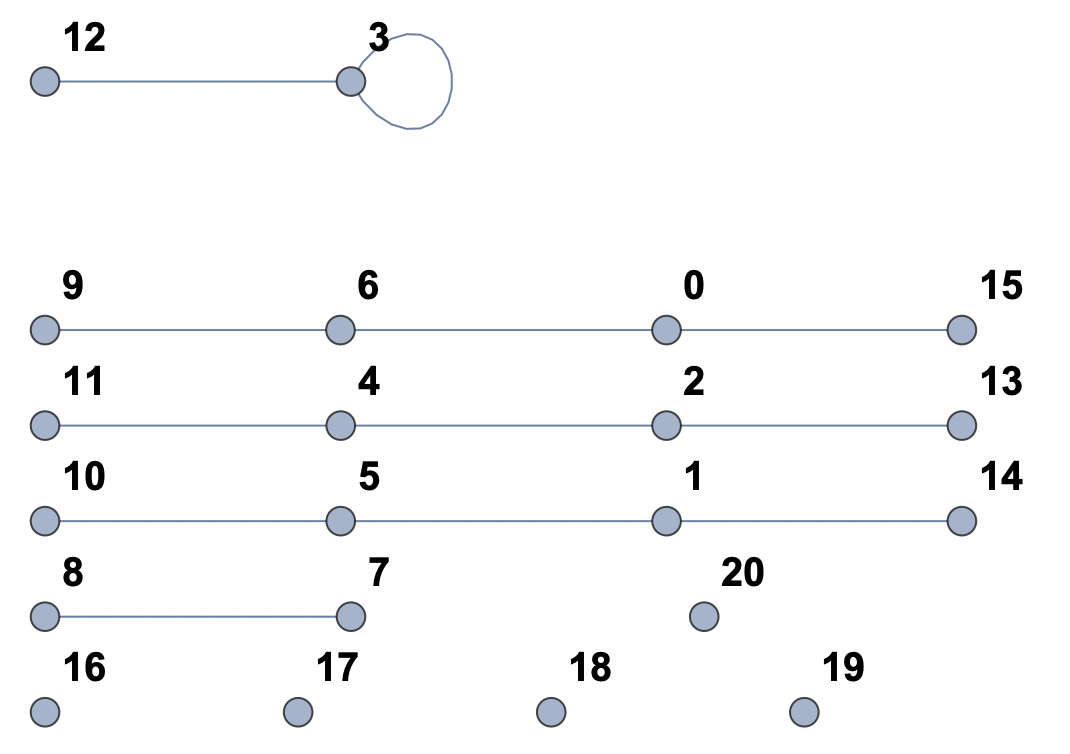}
\caption{Condition graph induced on $V = [0,20]$ by $F=\{6,15\}$.}
\label{fig: condition37}
\end{figure}

See Figure ~\ref{fig: condition37} for the condition graph $G_{6,15}$ induced on $n=21$. Explicitly, note that if $i<j$, any vertex $k_1$ with $k_1>j$ must be isolated in $G_{i,j}$, since there is no $k_2$ so that $k_1+k_2\in \{i,j\}$. We also allow loops, and the number of loops we see will be exactly the number of even elements of $F$. In our example above only $6$ is even.
By construction, as \cite{LMO} explained, viewing our vertices as the integers $[0,n-1]$, we have a bijection between edges and pairs of elements whose sum belongs to $F$. If we suppose that $F \cap (A+A)=\varnothing$, then, for each pair of elements whose sum belongs to $F$, at least one of the pair must be excluded from $A$. The corresponding criteria in the condition graph $G_F$ is that each edge must be incident to a vertex which corresponds to an integer missing from $A$. That is, $F \cap (A+A) = \varnothing$ exactly when $A$ is an independent set on $G_{i,j}$ (recall an independent set of a graph is a set of vertices with no edges shared between any pair of vertices). We therefore find the following, which is equivalent to Lemma 2.1 from \cite{LMO}. It is worth noting that \cite{LMO} stated this lemma in the language of vertex covers, but for the sake of our paper, we discuss the same idea with independent sets.

\begin{lemma}\label{lem:vertexcover}
For sets $V=[0,n-1]$ and $F \subseteq [0,2n-2]$, $\mathbb{P}(F \cap (A+A)=\varnothing)$ is the probability that we choose an independent set for the condition graph $G_F$ induced on $V$ by $F$.
\end{lemma}

We now find a closed form for $\mathbb{P}(i \text{ and } j \not\in A+A)$. By Lemma ~\ref{lem:vertexcover}, we only need to study the condition graph $G_{i,j}$ induced on $[0,n-1]$. From Proposition 3.1 of \cite{LMO}, each component in our condition graph $G_{i,j}$ is a path graph, where loops are also allowed. Since we only add isolated vertices to \cite{LMO}'s definition of a condition graph, the proof of their Proposition 3.1 applies to our condition graph. As we are interested in counting independent sets and there are no edges between different components, the behavior of each component is independent. That is, the probability of finding an independent set for the entire graph is the product of the probability of finding an independent set on each component. In this way, we reduce the problem at hand to computing the probability of finding an independent set on a path graph, which we do in the following lemma and corollary.
\begin{lemma}\label{lem:iandjnotinaplusa}
Let $S$ be a subset of the vertices of a path graph with no loops on the $n$ vertices $V$, with each vertex included in $S$ with probability $p$. If we set $a_n = \mathbb{P}(S \text{ {\rm is\ an\ independent\ set}})$, then
\begin{equation}\label{eqn:recurrencesolve}
 a_n\ =\ \frac{(g(p)-1-p)\,(1-p-g(p))^n + (g(p)+1+p)\,(1-p+g(p))^n}{2^{n+1}\,g(p)},
\end{equation}
where $g(p) = \sqrt{1+2p-3p^2}$.
\end{lemma}
\begin{proof}
    The proof follows from a simple recurrence relation. We see that $a_1=1$, as this path does not have loops, so we cannot have an edge if only one vertex exists. Also, $a_2 = 1-p^2$, as the only case in which we do not get an independent set is when both vertices $v_1,v_2$ are in $S$. This happens with probability $p^2$, as each event is independent. We now find a recurrence relation to compute $a_n$.

    \begin{figure}[ht]
   \begin{tikzpicture}[scale=.7,auto=left,every node/.style={circle,fill=black!20,minimum size=12pt,inner sep=2pt}] 

 \node (n0) at (0,0)  {$\ \ v_1\ \  $};
 \node (n1) at (2.5,0)  {$\ \ \dots\  \  $};
 \node (n2) at (5,0)  {$v_{n-2}$};
 \node (n3) at (7.5,0)  {$v_{n-1}$};
 \node (n4) at (10,0)  {$\ \ v_n \ \ $};

   \foreach \from/\to in {n0/n1,n1/n2,n2/n3,n3/n4}
    \draw (\from) -- (\to);

 \end{tikzpicture}
 \caption{A path with $n$ vertices; we may count indenpendent sets recursively by handling the behavior of the $n$\textsuperscript{th} vertex.}
 \label{fig:path_recurrence}
\end{figure}
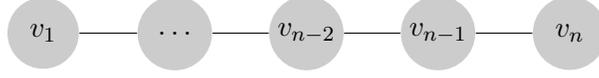

    In Figure \ref{fig:path_recurrence} we see that if $v_n \not\in S$, then we can recurse on the remaining $n-1$ vertices, as the edge connecting $v_n$ to $v_{n-1}$ has an incident vertex in the complement of $S$. This event has probability $(1-p)\,a_{n-1}$ of occurring. However, if $v_n\in S$, we must necessarily have $v_{n-1}\not\in S$ for $S$ to be an independent set, and then we can recurse on the remaining $n-2$ vertices. This event has probability $p\,(1-p)\,a_{n-2}$ of occurring. So, we find that
    \[
        a_n \ = \ (1-p)\,a_{n-1}+p\,(1-p)\,a_{n-2},
    \]
    with $a_1=1$ and $a_2 = 1-p^2$. Solving this gives us \eqref{eqn:recurrencesolve}, as desired.
\end{proof}

\begin{corollary}
Let $S$ be a subset of the vertices of a path graph on $n$ vertices $V$ with a single loop on the $n$th vertex, as seen in Figure \ref{fig:path_loop}, with each vertex included in $S$ with probability $p$. Then we have $\mathbb{P}(S \text{ {\rm is\ an\ independent\ set}}) = (1-p)a_{n-1}$, where $a_n$ is as defined in Lemma \ref{lem:iandjnotinaplusa}.
\end{corollary}

   \begin{figure}[ht!]
   \begin{tikzpicture}[scale=.7,auto=left,every node/.style={circle,fill=black!20,minimum size=12pt,inner sep=2pt}] 

 \node (n0) at (0,0)  {$\ \ v_1\ \  $};
 \node (n1) at (2.5,0)  {$\ \ \dots\  \  $};
 \node (n2) at (5,0)  {$v_{n-2}$};
 \node (n3) at (7.5,0)  {$v_{n-1}$};
 \node (n4) at (10,0)  {$\ \ v_n \ \ $};
 \node (n5) at (12,0) [circle,fill=blue!20,minimum size=0pt,inner sep=0pt] {};
   \foreach \from/\to in {n0/n1,n1/n2,n2/n3,n3/n4}
    \draw (\from) -- (\to);

\draw [black, -] (n4) to [out=45,in=90] (n5);
\draw [black, -] (n4) to [out=315,in=270] (n5);
 \end{tikzpicture}
 \caption{A path with $n$ vertices and a loop on one end.}
 \label{fig:path_loop}
\end{figure}
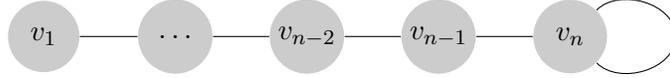

\begin{proof}
For $S$ to be an independent set, we cannot have the vertex with a loop in $S$. This occurs with probability $1-p$. Once we remove this vertex, we now have a path graph of $n-1$ vertices with no loops. Thus, by Lemma \ref{lem:iandjnotinaplusa}, $\mathbb{P}(S \text{ {\rm is\ an\ independent\ set}}) = (1-p)a_{n-1}$.
\end{proof}

Now, to find $\mathbb{P}(i\text{ and } j\not\in A+A)$, we only need to describe $G_{i,j}$ as the union of disjoint path components. Fortunately, \cite{LMO} derived formulas for the number and size of path graphs (their Proposition 3.5). Using these, we find

\begin{proposition}\label{prop:p(i,j)}
Consider $i,j$ such that $i < j$.\\
For $i, j$ both odd:
\begin{equation}
\mathbb{P}(i \mbox{ and }j \not \in A+A)\ =\  a_{q}^s\, a_{q +2}^{s'},
\end{equation}
where
\begin{eqnarray}
q & \ = \ & 2\left\lceil \frac{i+1}{j-i} \right\rceil, \nonumber\\
s &=& \frac{1}{2}\left((j-i)\left\lceil \frac{i+1}{j-i} \right\rceil - (i +1)\right), \nonumber\\
s' &=& \frac{1}{2}\left(j+1 -(j-i)\left\lceil \frac{i+1}{j-i} \right\rceil \right).
\end{eqnarray}

For $i$ even, $j$ odd:
\begin{equation}
\mathbb{P}(i \mbox{ and }j \not \in A+A)\ =\ (1-p)\,a_o\, a_{q}^s\, a_{q +2}^{s'},
\end{equation}
where
\begin{eqnarray}
o &\ = \ & 2 \left \lceil \frac{i/2+1}{j-i} \right \rceil -1, \nonumber\\
q &=& 2\left\lceil \frac{i+1}{j-i} \right\rceil, \nonumber\\
s &=& \frac{1}{2}\left((j-i-1)\left\lceil \frac{i+1}{j-i} \right\rceil - (i +1)+o\right), \nonumber\\
s' &=& \frac{1}{2}\left(j -(j-i-1)\left\lceil \frac{i+1}{j-i} \right\rceil - o \right).
\end{eqnarray}

For $i$ odd, $j$ even:
\begin{equation}
\mathbb{P}(i \mbox{ and }j \not \in A+A)\ =\ (1-p)\,a_{o'}\, a_{q}^s\, a_{q +2}^{s'},
\end{equation}
where
\begin{eqnarray}
o' &\ = \ &2 \left \lceil \frac{j/2+1}{j-i} \right \rceil -2, \nonumber\\
q &=& 2\left\lceil \frac{i+1}{j-i} \right\rceil, \nonumber\\
s &=& \frac{1}{2}\left((j-i-1)\left\lceil \frac{i+1}{j-i} \right\rceil - (i +1)+ o'\right), \nonumber\\
s' &=& \frac{1}{2}\left(j -(j-i-1)\left\lceil \frac{i+1}{j-i} \right\rceil - o' \right).
\end{eqnarray}

For $i, j$ both even:
\begin{equation}
\mathbb{P}(i \mbox{ and }j \not \in A+A)\ =\ (1-p)^2\,a_{o} \,a_{o'}\,a_{q}^s\, a_{q +2}^{s'},
\end{equation}
where
\begin{eqnarray}
o &\ =\ & 2 \left \lceil \frac{i/2+1}{j-i} \right \rceil -1, \nonumber\\
o' &=& 2 \left \lceil \frac{j/2+1}{j-i} \right \rceil -2, \nonumber\\
q &=& 2\left\lceil \frac{i+1}{j-i} \right\rceil, \nonumber\\
s &=& \frac{1}{2}\left((j-i-2)\left\lceil \frac{i+1}{j-i} \right\rceil - (i +1)+ o + o'\right), \nonumber\\
s' &=& \frac{1}{2}\left(j-1 -(j-i-2)\left\lceil \frac{i+1}{j-i} \right\rceil - o - o'\right).
\end{eqnarray}
\end{proposition}

The proof of Proposition ~\ref{prop:p(i,j)} is structurally identical to that of Proposition 3.5 of \cite{LMO}, as we have already shown independence of path graphs, so we must show how to obtain the number of path graphs and their size, which was done in \cite{LMO}. We briefly explain the role of each parameter.

\begin{figure}[ht]
    \centering
    \includegraphics[width=0.98\textwidth]{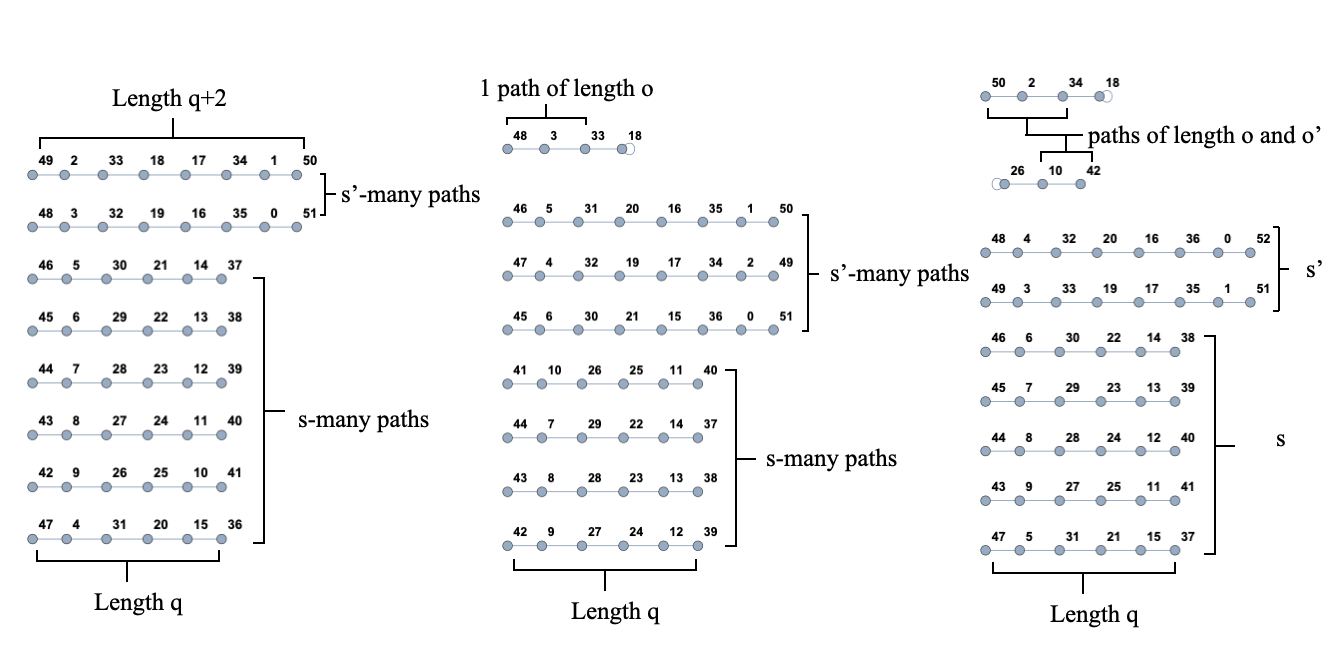}
    \caption{From left to right, $G_{35,51},G_{36,51}$, and $G_{36,52}$}
    \label{fig:graph_decomp}
\end{figure}

We see in Figure \ref{fig:graph_decomp} that we see loops based on the number of $\{i,j\}$ which are even. In the simplest case, both are odd, and \cite{LMO} showed that $G_{i,j}$ consists only of paths of length $q$ and $q+2$, each appearing with frequency $s$ and $s'$ (where $q,s$, and $s'$ are functions of $i$ and $j$). If, for example, $i$ is even and $j$ is odd, then \cite{LMO} showed that one path forms a loop with $i/2$, with the rest of this path has length $o$, while the remaining paths still appear of length $q$ or $q+2$ with certain frequencies. If both $i$ and $j$ are even we see the appearance of two loops, attached to smaller paths of length $o$ and $o'$. Since we require an independent set, we cannot include any vertex with a loop (which is adjacent to itself), and therefore we find the terms $(1-p)$ and $(1-p)^2$ in Proposition ~\ref{prop:p(i,j)} when one or both of $i,j$ are even, and the terms $a_i$ appear according to the frequency of the path of that length in $G_{i,j}$, as calculated in \cite{LMO}.
We now bound $\mathbb{P}(i\text{ and } j\not\in A+A)$. We note, from Equation \eqref{eqn:recurrencesolve}, that if $n$ is even, then
\begin{equation}\label{eqn:recurrencesolvebound}
    a_n\ \leq\ \frac{(g(p)+1+p)\,(1-p+g(p))^n}{2^{n+1}\,g(p)}.
\end{equation}
Since $q$ and $q+2$ are always even, for odd $i,j$, we have
\begin{eqnarray}\label{eqn:decay}
\mathbb{P}(i\text{ and }j \not\in A+A) &=& a_{q}^s \,a_{q +2}^{s'} \nonumber\\
&\leq& \pn{\frac{(g(p)+1+p)\,(1-p+g(p))^q}{2^{q+1}\,g(p)}}^s\pn{\frac{(g(p)+1+p)\,(1-p+g(p))^{q+2}}{2^{q+3}\,g(p)}}^{s'}\nonumber\\
&=&\frac{(g(p)+1+p)^{s+s'}\,(1-p+g(p))^{qs+(q+2)s'}}{2^{(q+1)\,s+(q+3)\,s'}\,g(p)^{s+s'}}\nonumber\\
&=&\bigg(\frac{g(p)+1+p}{2g(p)}\bigg)^{s+s'}\bigg(\frac{1-p+g(p)}{2}\bigg)^{qs+(q+2)s'}\nonumber\\
&=&\bigg(\frac{g(p)+1+p}{2g(p)}\bigg)^{\frac{j-i}{2}}\bigg(\frac{1-p+g(p)}{2}\bigg)^{j+1},
\end{eqnarray}
where the last equality comes from (3.18) in \cite{LMO}. We can use Proposition ~\ref{prop:p(i,j)} to show \eqref{eqn:decay} holds for all $i,j$.

\section{Expected Value}\label{sec:expecvalue}

To compute $\mathbb{E}[|A+A|]$, we see that

\begin{eqnarray}
\mathbb{E}[|A+A|]&=&\sum_{A \subseteq \{0,\dots,n-1\}}|A+A| \cdot \mathbb{P}(A)\nonumber\\
&=&\sum_{r\,=\,0}^{n}\binom{n}{r}\ p^r q^{n-r}\sum_{i\, =\, 0}^{2n-2}\sum_{\substack{A \subseteq \{0,\dots,  n-1\},|A|=r\\i\ \in\ A}}1\nonumber\\
&=&\sum_{r\,=\,0}^{n}\binom{n}{r}\ p^r q^{n-r}\sum_{i\,=\,0}^{2n-2}\mathbb{P}(i \in A+A \ | \ |A|=r).\label{eqn:expval}
\end{eqnarray}

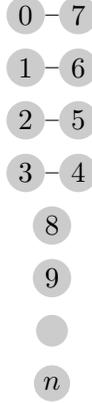
\begin{figure}[ht]
   \begin{tikzpicture}[scale=.7,auto=left,every node/.style={circle,fill=black!20,minimum size=12pt,inner sep=2pt}] 
 \node (n0) at (0,0)  {$0$};
 \node (n1) at (1,0)  {$7$};
 \node (n2) at (0,-1)  {$1$};
 \node (n3) at (1,-1)  {$6$};
 \node (n4) at (0,-2)  {$2$};
 \node (n5) at (1,-2)  {$5$};
 \node (n6) at (0,-3)  {$3$};
 \node (n7) at (1,-3)  {$4$};
 \node at (0.5,-4)  {$8$};
 \node at (0.5,-5)  {$9$};
 \node at (0.5,-6)  {};
 \node at (0.5,-7) {$n$};

  \foreach \from/\to in {n0/n1,n2/n3,n4/n5,n6/n7}
    \draw (\from) -- (\to);
 \end{tikzpicture}
\caption{Condition Graph for $\mathbb{P}( 7 \not\in A+A)$.}
\label{fig:condition_simple_example}
\end{figure}

Now we compute $\mathbb{P}(i \in A+A \ | \ |A|=r)=1-\mathbb{P}(i \not \in A+A \ | \ |A|=r)$. To compute this probability, we refer to the condition graph $G_i$ induced on $[0,n-1]$. If we assume\footnote{The $i >n-1$ case is identical after reflection about $n-1$.} that $i \leq n-1$, this graph has $n$ vertices and either $\frac{i+1}{2}$ disjoint simple edges (if $i$ is odd) or $\frac{i}{2}$ disjoint simple edges with one additional loop (if $i$ is even). See Figure ~\ref{fig:condition_simple_example} for a visualization.

By Lemma ~\ref{lem:vertexcover}, the event $i \not \in A+A$ corresponds to when the elements in $A$ form an independent set of $G_i$. Since we are conditioning that $|A|=r$, we must count the number of ways that the $r$ elements may be chosen so that they form an independent set. Then we obtain the following:

\begin{lemma}\label{lem:expectedvaluehelper}
Let $i\in [0,2n-2]$ be given. Then, for all $0\leq r\leq n$,

\begin{align}
    \mathbb{P}(i \not \in A+A \ | \ |A|=r)& \ = \ \frac{\text{\# {\rm ways\ to\ place\ $r$\ vertices\ on\ disjoint\ edges\ to\ get\ an\ independent\ set}}}{\text{\# {\rm ways\ to\ choose\ $r$\ vertices\ from\ $n$ vertices}}}\nonumber\\
    &\ = \ \begin{cases}\frac{\displaystyle\sum_{k\, =\, 0}^{\frac{i+1}2}2^k\binom{\frac{i+1}{2}}{k}\binom{n-i-1}{r-k}}{\displaystyle\binom{n}{r}} &\text{ {\rm for}\ i \text{\rm odd},}\nonumber\\
    \frac{\displaystyle\sum_{k\, =\, 0}^{\frac{i}2}2^k\binom{\frac{i}{2}}{k}\binom{n-i-1}{r-k-1}}{\displaystyle\binom{n}{r}}&\text{ {\rm for}\ i \text{{\rm even.}}}
    \end{cases}\label{eqn:p_k_notin_s_cond}
\end{align}

\end{lemma}

\begin{proof}
The derivation for the odd cases is as follows; the even cases can be handled similarly. We divide $G_i$ into two components; $J$ containing the $\frac{i+1}{2}$ disjoint edges and $H$ containing $n-i-1$ isolated vertices. To count the number of independent sets, we denote by $k$ the number of our $r$ vertices which are placed in $J$. First, having fixed $k$, we must choose $r-k$ vertices from the $n-i-1$ vertices in $H$, with no edge restrictions. Second, inside $J$ we must determine from which edges we will choose a vertex, as we cannot choose two vertices that share an edge. This gives us a factor of $\binom{\frac{i+1}{2}}{k}$. Finally, those edges in $J$ which do take a vertex may take either the left or the right vertex, which gives a factor of $2^k$. We then divide by $\binom{n}{r}$, which is the number of ways to choose $r$ vertices from $n$ vertices.
\end{proof}

Then, if in \eqref{eqn:expval} we use the symmetry around $n-1$ to double terms and account for $n-1<i<2n-2$, we find that

\begin{equation}\label{eqn:expval3}
\mathbb{E}[|A+A|] \ = \ \sum_{r\,=\,0}^{n}\binom{n}{r}\ p^r q^{n-r}\left(2\sum_{i\,=\,0}^{n-2}(1-\mathbb{P}(i \not \in A+A \ | \ |A|=r))+(1-\mathbb{P}(n-1 \not \in A+A \ | \ |A|=r))\right).
\end{equation}
This proves Theorem ~\ref{thm:expectedvalue}, because Equation \eqref{eqn:expval3} is exactly the claim. \hfill $\Box$

\ \\
While this closed form is exact and easily approximated numerically, there are $O(n^3)$ sums to execute. We wish to place effective upper and lower bounds on this sum. First notice that
\begin{align}
   \begin{cases}\displaystyle\sum_{k\, =\, 0}^{\frac{i+1}2}2^k\binom{\frac{i+1}{2}}{k}\binom{n-i-1}{r-k} &\text{ for }i \text{ odd}\nonumber\\
    \displaystyle\sum_{k\, =\, 0}^{\frac{i}2}2^k\binom{\frac{i}{2}}{k}\binom{n-i-1}{r-k-1}&\text{ for }i \text{ even}
    \end{cases}\geq \begin{cases}\displaystyle\binom{n-\frac{i+1}{2}}{r} &\text{ for }i \text{ odd}\\
    \displaystyle\binom{n-\frac{i}{2}-1}{r}&\text{ for }i \text{ even.}
    \end{cases}
\end{align}

As discussed before, the left-hand side counts the number of independent sets using $r$ vertices on our graph $G$. The right-hand side undercounts the number of such independent sets, by first (in the odd case) assigning $\frac{i+1}{2}$ vertices corresponding to each edge, removing the assigned edges from consideration, and then choosing the $r$ vertices freely from the remaining $n-\frac{i+1}{2}$ vertices. Substituting this into \eqref{eqn:expval3}, we find that
\begin{eqnarray} \mathbb{E}[|A+A|] & \ \leq \ & \sum_{r=0}^{n} p^r q^{n-r}\bigg(2\sum_{i\,=\,0}^{n-2}\pn{\binom{n}{r}-\begin{cases}\binom{n-\frac{i+1}{2}}{r} &\text{ for }i \text{ odd}\\
    \binom{n-\frac{i}{2}-1}{r}&\text{ for }i \text{ even}
    \end{cases}} \nonumber \\
    & & \ \ + \ \binom{n}{r}-\begin{cases}\binom{n-\frac{n}{2}}{r} &\text{ for }n-1 \text{ odd}\\
    \binom{n-\frac{n-1}{2}-1}{r}&\text{ for }n-1 \text{ even}
    \end{cases}\bigg).
\end{eqnarray}

We organize these by collecting those terms of the form $\binom{n}{r}$ to find a binomial which necessarily sums to $1$. Specifically,
\begin{eqnarray}
\mathbb{E}[|A+A|]\ \leq\ \sum_{r=0}^{n} p^r q^{n-r}\binom{n}{r}\pn{2\pn{\sum_{i\,=\,0}^{n-2}1}+1}&-&\sum_{r\,=\,0}^{n}p^r q^{n-r}2\sum_{i\,=\,0}^{n-2}{\begin{cases}\binom{n-\frac{i+1}{2}}{r} &\text{ for }i \text{ odd}\\
    \binom{n-\frac{i}{2}-1}{r}&\text{ for }i \text{ even}
    \end{cases}}\nonumber\\
    &-&\sum_{r\,=\,0}^{n} p^r q^{n-r}\begin{cases}\binom{n-\frac{n}{2}}{r} &\text{ for }n-1 \text{ odd}\\
    \binom{n-\frac{n-1}{2}-1}{r}&\text{ for }n-1 \text{even.}
    \end{cases}\nonumber\\
\end{eqnarray}
The first sum over $r$ we see is binomial in $r$, and for a fixed value of $i$, gives us
\begin{align}
\sum_{r\,=\,0}^{n} p^r q^{n-r}\begin{cases}\binom{n-\frac{i+1}{2}}{r} &\text{ for }i \text{ odd,}\\
    \binom{n-\frac{i}{2}-1}{r}&\text{ for }i \text{ even}.
    \end{cases}
\end{align}
By factoring $q^{\frac{i+1}{2}}$ or $q^{\frac{i}{2}}$ out of this sum, we get, once again, a sum of probabilities of events under a binomial distribution which must sum to $1$. We omit the last term corresponding to $n-1$, as we seek an upper bound. Then
\begin{align}
\mathbb{E}[|A+A|]&\ \leq\ 2\sum_{i\,=\,0}^{n-2}1+1-\sum_{i\,=\,0}^{n-2}\begin{cases}q^{\frac{i+1}{2}} &\text{ for }i \text{ odd}\\
    q^{\frac{i+2}{2}}&\text{ for }i \text{ even}
    \end{cases}\nonumber\\
    &\ =\ 2n-1-2q\sum_{i\,=\,0}^{n-2}(\sqrt{q})^i\nonumber\\
    &\ =\ 2n-1-2q\frac{1-q^{\frac{n-1}{2}}}{1-\sqrt{q}}
\end{align}
as needed to prove the first statement of Theorem ~\ref{thm:expectedvaluebounds}.

To derive a lower bound, we first see that
\[
   \begin{cases}\displaystyle\sum_{k\, =\, 0}^{\frac{i+1}2}2^k\binom{\frac{i+1}{2}}{k}\binom{n-i-1}{r-k} &\text{ for }i \text{ odd}\nonumber\\
    \displaystyle\sum_{k\, =\, 0}^{\frac{i}2}2^k\binom{\frac{i}{2}}{k}\binom{n-i-1}{r-k-1}&\text{ for }i \text{ even}
    \end{cases}\ \leq\ \begin{cases}\displaystyle 2^{\frac{i+1}{2}}\binom{n-\frac{i+1}{2}}{r} &\text{ for }i \text{ odd}\\
    \displaystyle 2^{\frac{i}{2}}\binom{n-\frac{i}{2}-1}{r}&\text{ for }i \text{ even.}
    \end{cases}
\]

Similar to before, on the right-hand side we are over-counting the number of ways to create an independent set using $r$ vertices by now adding a factor of (in the odd case) $2^{\frac{i+1}2}$, which corresponds to the choice of which vertex we are excluding from consideration from each disjoint edge.

Then by substitution,
\begin{multline}
 \mathbb{E}[|A+A|]\ \geq\
 \sum_{r\,=\,0}^{n} p^r q^{n-r}2\sum_{i\,=\,0}^{n-2}\Bigg(\binom{n}{r}-\begin{cases}2^{\frac{i+1}{2}}\binom{n-\frac{i+1}{2}}{r} &\text{ for }i \text{ odd}\\
    2^{\frac{i}{2}}\binom{n-\frac{i}{2}-1}{r}&\text{ for }i \text{ even}
    \end{cases}\\ + \binom{n}{r}-\begin{cases}2^{\frac{n}{2}}\binom{n-\frac{n}{2}}{r} &\text{ for }n-1 \text{ odd}\\
    2^{\frac{n-1}{2}}\binom{n-\frac{n-1}{2}-1}{r}&\text{ for }n-1 \text{ even}
    \end{cases}\Bigg).
\end{multline}

If we distribute this sum into individual components, we get
\begin{multline}
 \mathbb{E}[|A+A|]\ \geq\
 \sum_{r=0}^{n} p^r q^{n-r}\Bigg(2\pn{\sum_{i\,=\,0}^{n-2}\binom{n}{r}}+2\sum_{i\,=\,0}^{n-2}\pn{-\begin{cases}2^{\frac{i+1}{2}}\binom{n-\frac{i+1}{2}}{r} &\text{ for }i \text{ odd}\\
    2^{\frac{i}{2}}\binom{n-\frac{i}{2}-1}{r}&\text{ for }i \text{ even}
    \end{cases}}
    \\ + \pn{\binom{n}{r}-\begin{cases}2^{\frac{n}{2}}\binom{n-\frac{n}{2}}{r} &\text{ for }n-1 \text{ odd}\\
    2^{\frac{n-1}{2}}\binom{n-\frac{n-1}{2}-1}{r}&\text{ for }n-1 \text{ even}
    \end{cases}}\Bigg).
\end{multline}
Exchanging the order of summation gives
\begin{multline}
\mathbb{E}[|A+A|]\ \geq\
 \sum_{i\,=\,0}^{n-2} \bigg(2\sum_{r=0}^{n}\pn{p^r q^{n-r}\binom{n}{r}}+2\sum_{r=0}^{n}\pn{-p^r q^{n-r}\begin{cases}2^{\frac{i+1}{2}}\binom{n-\frac{i+1}{2}}{r} &\text{ for }i \text{ odd}\\
    2^{\frac{i}{2}}\binom{n-\frac{i}{2}-1}{r}&\text{ for }i \text{ even}
    \end{cases}}\bigg)\\ + \sum_{r=0}^{n} p^r q^{n-r}\pn{\binom{n}{r}-\begin{cases}2^{\frac{n}{2}}\binom{n-\frac{n}{2}}{r} &\text{ for }n-1 \text{ odd}\\
    2^{\frac{n-1}{2}}\binom{n-\frac{n-1}{2}-1}{r}&\text{ for }n-1 \text{ even}
    \end{cases}}.
\end{multline}
Now the first sum over $r$ is $1$, a binomial sum. That is,
\begin{multline}\label{eqn:lb_expansion}
 \mathbb{E}[|A+A|]\ \geq\
 \sum_{i\,=\,0}^{n-2} \left(2-2\sum_{r=0}^{n}\pn{p^r q^{n-r}\begin{cases}2^{\frac{i+1}{2}}\binom{n-\frac{i+1}{2}}{r} &\text{ for }i \text{ odd}\\
    2^{\frac{i}{2}}\binom{n-\frac{i}{2}-1}{r}&\text{ for }i \text{ even}
    \end{cases}}\right)\\ + \sum_{r=0}^{n} p^r q^{n-r}\pn{\binom{n}{r}-\begin{cases}2^{\frac{n}{2}}\binom{n-\frac{n}{2}}{r} &\text{ for }n-1 \text{ odd}\\
    2^{\frac{n-1}{2}}\binom{n-\frac{n-1}{2}-1}{r}&\text{ for }n-1 \text{ even}
    \end{cases}}.
\end{multline}
Let us consider, for a moment, the remaining terms inside. They depend on both $r$ and $i$, but, for a fixed value of $i$, they resemble a binomial sum. That is, for fixed $i$,
\begin{align}
\sum_{r\,=\,0}^{n} p^r q^{n-r}\begin{cases}2^{\frac{i+1}{2}}\binom{n-\frac{i+1}{2}}{r} &\text{ for }i \text{ odd}\nonumber\\
    2^{\frac{i}{2}}\binom{n-\frac{i}{2}-1}{r}&\text{ for }i \text{ even}
    \end{cases} \ = \ \begin{cases}(2q)^{\frac{i+1}{2}} &\text{ for }i \text{ odd}\nonumber\\
    (2q)^{\frac{i+2}{2}}&\text{ for }i \text{ even.}
    \end{cases}
\end{align}
Thus we have
\begin{eqnarray}
& & \mathbb{E}[|A+A|]\ \geq\ \nonumber\\
& & \sum_{i\,=\,0}^{n-2} \bigg(2-2\begin{cases}(2q)^{\frac{i+1}{2}} &\text{ for }i \text{ odd}\nonumber\\
    (2q)^{\frac{i+2}{2}}&\text{ for }i \text{ even.}
    \end{cases}\bigg)+\sum_{r=0}^{n} p^r q^{n-r}\pn{\binom{n}{r}-\begin{cases}2^{\frac{n}{2}}\binom{n-\frac{n}{2}}{r} &\text{ for }n-1 \text{ odd}\\
    2^{\frac{n-1}{2}}\binom{n-\frac{n-1}{2}-1}{r}&\text{ for }n-1 \text{ even}
    \end{cases}}. \nonumber
\end{eqnarray}
Now consider those terms independent of $i$. The first, $\binom{n}{r}$, is once again a simple binomial sum. The last term corresponding to $n-1$ may be handled the same way by factoring out $(2q)^{\frac{n}{2}}$ or $(2q)^{\frac{n-1}{2}}$, depending on parity; for a lower bound we choose to subtract the larger $(2q)^{\frac{n-1}{2}}$, and find that, using our assumption that $p>1/2$ implies $q < 1/2$, so we may apply the geometric series formulae to obtain

\begin{align}
\mathbb{E}[|A+A|]&\ \geq\ 2\sum_{i\,=\,0}^{n-2}\pn{1-\begin{cases}(2q)^{\frac{i+1}{2}} &\text{ for }i \text{ odd}\\
    (2q)^{\frac{i+2}{2}}&\text{ for }i \text{ even}
    \end{cases}}+1-(2q)^{\frac{n-1}{2}}\nonumber\\
    & \ = \ 2n-1-2q\sum_{i\,=\,0}^{n-2}(\sqrt{2q})^i-(2q)^{\frac{n-1}{2}}\nonumber\\
    & \ = \ 2n-1-\frac{2q}{1-\sqrt{2q}}-(2q)^{\frac{n-1}{2}},
\end{align} completing the proof of Theorem ~\ref{thm:expectedvaluebounds}. \hfill $\Box$


\section{Variance}\label{sec:variance}

We now find the variance. Recall
\begin{equation}\label{eqn:varianceidentity}
    \Var(|A+A|) \ = \ \mathbb{E}[|A+A|^2] - \mathbb{E}[|A+A|]^2.
\end{equation}
In the previous section, we computed $\mathbb{E}[|A+A|]$, so we need only to determine $\mathbb{E}[|A+A|^2]$. We apply the same technique used to compute the expected value and condition each probability on the size of $A$. We have
\begin{align}
&\mathbb{E}[|A+A|^2] \ = \ \sum_{A\,\subseteq\,\{0,\dots,n-1\}} |A+A|^2 \cdot \mathbb{P}(A)\nonumber\\
&\ \ \ = \ \sum_{r\,=\,0}^{n}\binom{n}{r}\ p^r q^{n-r}\sum_{A\, \subseteq\, [0,n-1]\,,\,|A|\,=\,r}|A+A|^2\nonumber\\
&\ \ \ = \ \sum_{r\,=\,0}^{n}\binom{n}{r}\ p^r q^{n-r}\sum_{0\,\leq \,i,j \,\leq\, 2n-2}\sum_{\substack{A \,\subseteq\, [0,n-1]\,,\,|A|\,=\,r\\i,j\,\in\, A+A}}1\nonumber\\
&\ \ \ = \ \sum_{r\,=\,0}^{n}\binom{n}{r}p^r q^{n-r}\sum_{0\leq i,j \leq 2n-2}\mathbb{P}(i\text{ and }j  \in A+A \ | \ |A|=r)\nonumber\\
&\ \ \ = \ \sum_{r\,=\,0}^{n}\binom{n}{r}p^r q^{n-r}\pn{2\sum_{0\,\leq\, i<j \,\leq\, 2n-2}\mathbb{P}(i\text{ and }j  \in A+A\ | \ |A|=r)+\sum_{0\,\leq\, i\, \leq\, 2n-2}\mathbb{P}(i  \in A+A \ | \ |A|=r)}.
\end{align}

Similarly to the expected value, we compute $$\mathbb{P}(i\text{ and } j \in A+A \ | \ |A|=r)\ =\ 1-\mathbb{P}(i\text{ and } j \not \in A+A \ | \ |A|=r).$$ Once again, this reduces to a question about independent sets. In Proposition ~\ref{prop:p(i,j)}, we state formulas for the number, and size, of paths in the dependency graph $G$ associated to $i$ and $j$. We choose $r$ elements to be in $A$, and seek to compute the number of independent sets.

Unlike the dependency graph used for expected value, for a single $i$ here we have many options for distributing our $r$ chosen vertices. We attack this program in generality, and derive a solution which can then take, as input, the number and size of paths we know are in $G_i$. Suppose we wish to compute the number of independent sets on a graph consisting of $m$ paths, each of length $\ell_i$ for $1 \leq i \leq m$, with the remaining $n-\sum_{i\,=\,1}^{m}\ell_i$ vertices isolated. Then, given the number of vertices distributed to each path, we may compute the number of afforded independent sets. Summing over all such possible distribution schemes, we find the total number of independent sets. We state two lemmas that will be important in computing $\mathbb{P}(i\text{ and } j\not\in A+A \mid |A|=r)$.

\begin{lemma}\label{lemma:hockeystick}
Given a path graph $G$ of $\ell$ vertices, the number of ways to create an independent set using exactly $r$ vertices is
\[
    f(r,\ell) = \binom{\ell-r+1}{r}.
\]
\end{lemma}

\begin{proof}
We prove this by induction on $r+\ell = n$. When $n=0$, we note that $r=\ell=0$, so $f(0,0)$ trivially holds. We now suppose $f(r,\ell) =\binom{\ell-r+1}{r} $ holds for all non-negative integers at most $n$. First notice that given an independent set, there is a unique integer $0 \leq i \leq \ell$, so that the first $i$ vertices in the path are not taken to be in the independent set. We therefore know that the $(i+1)$st vertex is in the independent set and that we cannot choose the $(i+2)$nd vertex to be in the independent set. This gives us the recurrence relation $f(r,\ell) = \sum_{i=0}^{\ell} f(r-1,\ell-i-2)$. As for all $0\leq i\leq \ell$, $r-1 + \ell -i-2 \leq r+\ell = n$, the inductive hypothesis implies $f(r,\ell) = \sum_{i=0}^{\ell} \binom{\ell-r-i}{r-1}$. The hockey-stick (or Christmas stocking) identity then implies $f(r,\ell) = \binom{\ell-r+1}{r}.$
\end{proof}

\begin{lemma}\label{lem:pathcoverforvariance}
Let $G$ be a graph consisting of $n$ vertices with $m$ disjoint paths, with lengths $\ell_i$ for $1 \leq i \leq m$, and $t = n - \sum_{i\,=\,1}^m \ell_i$ isolated vertices. Then the number of independent sets of $G$ using exactly $r$ vertices is equal to
\[\sum_{\substack{r_0,r_1,\dots,r_m\in\, \mathbb{N}_0\\r_0\,+\,r_1\,+\,\cdots\,+\,r_m\, =\, r}}\binom{t}{r_0}\prod_{i\,=\,1}^{m}\binom{\ell_i-r_i+1}{r_i}.
\]
\end{lemma}

\begin{proof}
We must distribute the $r$ vertices amongst the pieces of our graph. This is exactly the internal sum. The $\binom{t}{r_0}$ term controls how many ways we may place those vertices in the edge-less block. Each of the $f(r_i,\ell_i)=\binom{\ell_i-r_i+1}{r_i}$ terms controls how many ways we may place the $r_i$ vertices in the path of length $\ell_i$ in order to obtain an independent set, as was shown in Lemma \ref{lemma:hockeystick}.
\end{proof}

We want to use Lemma ~\ref{lem:pathcoverforvariance} to compute $\mathbb{P}(i\text{ and } j \not\in A+A \ | \ |A| = r)$, in conjunction with Proposition ~\ref{prop:p(i,j)}, by plugging in the lengths and number of these paths. We find the following proposition.

\begin{proposition}\label{prop:variancehelper}
Let $i,j\in[0,2n-2]$ with $i < j$. Letting $P_r(i,j) =\mathbb{P}(i \text{ and } j \not\in A+A\ | \ |A|=r) $ for $i, j$ both odd, we get
\[
   \binom{n}{r}\, P_r(i,j) \ =\ \sum_{\substack{r_0,r_1,\dots,r_m\in\, \mathbb{N}_0\\r_0\,+\,r_1\,+\,\dots\,+\,r_m\, =\, r}}\binom{t}{r_0}\prod_{i=1}^{s}\binom{q-r_i+1}{r_i} \prod_{i=s+1}^{s+s'}\binom{q+2-r_i+1}{r_i},
\]
where $s, s'$ and $q$ are as defined in Proposition ~\ref{prop:p(i,j)}, $m = s+s'$ and $t = n - (qs + (q+2)s')$.\\

For $i$ even and $j$ odd:
\[
    \binom{n}{r}\,P_r(i,j) \ =\ \sum_{\substack{r_0,r_1,\dots,r_m\in\, \mathbb{N}_0\\r_0\,+\,r_1\,+\,\dots\,+\,r_m\, =\, r}}\binom{t}{r_0}\binom{o-r_m+1}{r_m}\prod_{i=1}^{s}\binom{q-r_i+1}{r_i} \prod_{i=s+1}^{s+s'}\binom{q+2-r_i+1}{r_i},
\]
where $s, s',o$ and $q$ are as defined in Proposition ~\ref{prop:p(i,j)}, $m = s+s'+1$ and $t = n - (qs + (q+2)s'+o)$.\\

For $i$ odd and $j$ even:
\[
    \binom{n}{r}\,P_r(i,j)\ =\ \sum_{\substack{r_0,r_1,\dots,r_m\in\, \mathbb{N}_0\\r_0\,+\,r_1\,+\,\dots\,+\,r_m\, =\, r}}\binom{t}{r_0}\binom{o'-r_m+1}{r_m}\prod_{i=1}^{s}\binom{q-r_i+1}{r_i}\prod_{i=s+1}^{s+s'}\binom{q+2-r_i+1}{r_i},
\]
where $s, s',o'$ and $q$ are as defined in Proposition ~\ref{prop:p(i,j)}, $m = s+s'+1$ and $t = n - (qs + (q+2)s'+o')$.\\

For $i,j$ both even:
\begin{multline*}
    \binom{n}{r}\,P_r(i,j) \\ = \sum_{\substack{r_0,r_1,\dots,r_m\in\, \mathbb{N}_0\\r_0\,+\,r_1\,+\,\dots\,+\,r_m\, =\, r}}\binom{t}{r_0}\binom{o-r_{m-1}+1}{r_{m-1}}\binom{o'-r_m+1}{r_m}\prod_{i=1}^{s}\binom{q-r_i+1}{r_i} \prod_{i=s+1}^{s+s'}\binom{q+2-r_i+1}{r_i},
\end{multline*}
where $s, s', o, o'$ and $q$ are as defined in Proposition ~\ref{prop:p(i,j)}, $m = s+s'+2$ and $t = n - (qs + (q+2)s'+o+o')$.
\end{proposition}

We find that
\begin{equation}
    \mathbb{E}[|A+A|^2] \ =\ \sum_{r\,=\,0}^n \binom{n}{r}p^r q^{n-r}\Big(2\sum_{0\leq i < j \leq 2n-2} 1-P_r(i,j) + \sum_{0\leq i \leq 2n-2} 1-P_r(i)  \Big),
\end{equation}
where $P_r(i,j) = \mathbb{P}(i\text{ and } j \not \in A+A \ | \ |A|=r)$ and $P_r(i) = \mathbb{P}(i \not \in A+A \ | \ |A|=r)$. We calculated $P_r(i,j)$ in Proposition ~\ref{prop:variancehelper}, and $P_r(i)$ in Lemma ~\ref{lem:expectedvaluehelper}. Since we have already calculated $\mathbb{E}[|A+A|]$ with Theorem ~\ref{thm:expectedvalue}, we have
\begin{equation}
    \Var(|A+A|)\ =\ \sum_{r\,=\,0}^n \binom{n}{r}p^r q^{n-r}\Big(2\sum_{0\leq i < j \leq 2n-2} 1-P_r(i,j) + \sum_{0\leq i \leq 2n-2} 1-P_r(i)  \Big) - \mathbb{E}[|A+A|]^2,
\end{equation}
which proves Theorem ~\ref{thm:variance}. \hfill $\Box$

\section{Divot Computations}\label{sec:divot}

In this section we prove Theorem ~\ref{thm:divot}. Fringe analysis has historically been the most successful technique for estimating probabilities of missing certain numbers of sums, and this is the method we follow. The technique grows out of the observation that sumsets usually have fully populated centers: there is a very low probability that an element from the bulk center of $[0,2n-2]$ is missing.\footnote{If each element of $[0, n-1]$ is chosen with probability $p$, the number of elements in $A$ is of size $pn$ with fluctuations of order $\sqrt{n}$. There are thus of order $p^2 n^2$ pairs of sums but only $2n-1$ possible sums, and most possible sums are realized.} When a suitable distance from the edge is chosen, this observation can be made precise. It follows that the number of missing sums is essentially controlled by the upper and lower fringes of the randomly chosen set. As long as they are short relative to the length of the whole set, the fringe behaviors at the top and bottom are independent and can be analyzed separately from the rest of the elements and each other. Furthermore, as long as they are reasonably sized (on the order of $30$ elements) a computer can exhaustively check all the possible fringe arrangements, and give exact data for the number of missing sums.

Our approach is to represent a general set $A$ as the union of a left, middle, and right part, where the left and right parts have fixed length $\ell$ and the middle size $n-2\ell$. Then, we establish good upper and lower bounds for $m_p(k)$, and use this to prove the existence of divots. First we develop some specialized notation for dealing with these fringe sets.

Fix a positive integer $\ell \le n/2$; this will be the ``fringe width''. Write $A=L\cup M\cup R$, where $L\subseteq [0,\ell-1], M\subseteq [\ell,n-\ell-1]$ and $R\subseteq [n-\ell,n-1]$. We look at $m_p(k)=\lim_{n\rightarrow\infty}m_{n\,;\,p}(k)$ for $k\in \mathbb{N}_0$ , the limiting distribution of missing sums.

Let $L_k$ be the event that $L+L$ misses exactly $k$ elements in $[0,\ell-1]$. Let $L^{a}_{k}$ be the event that $L+L$ misses exactly $k$ elements in $[0,\ell-1]$ and contains $[\ell,2\ell-a]$. Similar notations are applied to $R$; see below:
\begin{align}
L_k:\quad &|[0,\ell-1]\setminus (L+L)| \ = \  k,\nonumber\\
L_k^a:\quad &|[0,\ell-1]\setminus (L+L)| \ = \  k\text{ \rm and } [\ell,2\ell -a]\subseteq L+L,\nonumber\\
R_k:\quad &|[2n-\ell-1,2n-2]\setminus (R+R)| \ = \  k,\nonumber\\
R_k^a:\quad &|[2n-\ell-1,2n-2]\setminus (R+R)| \ = \  k\text{ \rm and } [2n-2\ell+a-2,2n-\ell-2]\subseteq R+R.
\end{align}
Next, let $\min L_k$ be the
minimal size of $L$ for which the event $L_k$ occurs, and similarly for the other events just defined; see below:
\begin{align}
&\min L_k\ = \ \min\{|L|:L_k\text{ \rm occurs}\},\nonumber\\
&\min R_k\ = \ \min\{|R|:R_k\text{ \rm occurs}\},\nonumber\\
&\min L^{a}_k\ = \ \min\{|L|:L_k^a\text{ \rm occurs}\},\nonumber\\
&\min R^{a}_k\ = \ \min\{|R|:R_k^a\text{ \rm occurs}\}.
\end{align}
Let
\begin{align}
\mathcal{M}_{L,k} \ &\ = \  \ \{L\subset [0,\ell-1]\ |\ L_k\text{ \rm occurs}\}, \nonumber\\
\mathcal{M}_{L,a,k} \ &\ = \  \ \{L\subset [0,\ell-1]\ |\ L^{a}_k\text{ \rm occurs}\}, \nonumber\\
\mathcal{M}_{R,k} \ &\ = \  \ \{R\subset [0,\ell-1]\ |\ R_k\text{ \rm occurs}\}, \nonumber\\
\mathcal{M}_{R,a,k} \ &\ = \  \ \{R\subset [0,\ell-1]\ |\ R^{a}_k\text{ \rm occurs}\},
\end{align}
and
\begin{align}
\tau\big(L^{a}_k\big)\ &\ = \ \ \min_{L\in\mathcal{M}_{L,k}}|L\cap [0,\ell-a+1]|,\nonumber\\
\tau\big(R^{a}_k\big)\ &\ = \ \ \min_{R\in\mathcal{M}_{R,k}}|R\cap [n-\ell+a-2,n-1]|.
\end{align}
By symmetry, for each $k\in\mathbb{N}_0$ we have $\min L_k=\min R_k$, $\min L^{a}_k=\min R^{a}_k$ and $\tau\big(L^{a}_k\big)=\tau\big(R^{a}_k\big)$. However, for clarity we will still distinguish these numbers despite that they are equal.

Our goal now is to prove effective upper and lower bounds for $m_{p}(k)$. We will frequently use the notation introduced above, that $A=L\cup M \cup R$.
\begin{itemize}
    \item To prove an upper bound on $m_{p}(k)$ in subsection \ref{upper}, we will quantitatively show that $[\ell,2n-\ell-2] \subset A+A$ with high probability. Then to upper bound $m_{p}(k)$ we consider the two cases, that either $[\ell,2n-\ell-2] \not \subset A+A$ (which is rare), or that all $k$ missing sums are in the fringes $[0,\ell-1]$ and $[2n-\ell-2,2n-2]$.
    \item To prove a lower bound on $m_{p}(k)$ in subsection \ref{lower} we must work a little harder. In particular, it is not enough to look at the ways for $k$ elements to be absent in the fringes $[0,\ell-1]$ and $[2n-\ell-2,2n-2]$ and then argue as before that $[\ell,2n-\ell-2] \subset A+A$ with high probability because these two events are not independent. Here our events $L_{k}^{a}$ will serve a vital role. We will look at the ways for $k$ elements to be absent from the fringe $[0,\ell-1]$ while $[\ell,2\ell-a]$ are all present. This will guarantee that at least some ($\tau(L_k^a$) many) elements are present in the fringes and can be used as ingredients to build the center bulk.
\end{itemize}
We will formalize these two arguments in Subsections \ref{upper} and \ref{lower} respectively, before harnessing normal computation to show the existence of a divot in Subsection \ref{ssec:computation}.

\subsection{An Upper Bound on $m_p(k)$}\label{upper}
In this section we place an upper bound on $m_p(k)$. First we show formally that our fringe events are independent.

\begin{lem}[Independence of the Fringes]\label{fringeindep} Pick a fringe width $\ell$. 
For any $k_1,k_2\in [0,\ell]$, the events
$L_{k_1}$ and $R_{k_2}$
are independent.
\end{lem}
\begin{proof}
The elements of $[0,n-1]$ are all chosen to be in $A$ or not independently. The only elements of $[0,n-1]$ which can contribute to $(A+A)\cap [0,\ell-1]$ are those in the interval $[0,\ell-1]$, since any larger element will sum to at least $\ell+0 \notin [0,\ell-1]$. Similarly, the only elements of $[0,n-1]$ which can contribute to $(A+A)\cap [2n-\ell-1,2n-2]$ are those in the interval $[n-\ell,n-1]$, since any smaller element will sum to at most $(n-\ell-1)+(n-1)\notin [2n-\ell-1,2n-2]$. Since $\ell\le n/2$,
\begin{equation}
[0,\ell-1]\cap [n-\ell,n-1]\ =\ \varnothing.
\end{equation}
Therefore the elements of $[0,\ell-1]\cap (L+L)$ and $[2n-\ell-1,2n-2]\cap (R+R)$ are independent, so in particular, the events $L_{k_1}$ and $R_{k_2}$ are independent.
\end{proof}

Next, we place an upper bound on the probability that the bulk of $A+A$ is missing at least one element.
\begin{lem}\label{lem:eventCbound}
Let $C$ denote the event that $\pn{\pn{[0,2n-2]}\setminus A+A}\cap [\ell,2n-\ell-2]\neq \varnothing$. Then
\begin{equation}\label{eqn:eventCbound}
\mathbb{P}(C)\ \le\
\dfrac{2(3q-q^2)(2q-q^2)^{\ell/2}}{(1-q)^2}
\end{equation}
\end{lem}

\begin{proof}
Because the event $C$ implies that $[\ell,2n-\ell-2]\not\subseteq A+A$,
\begin{align}\label{firstCineq}
\mathbb{P}(C)&\ \le\ \mathbb{P}([\ell,2n-\ell-2]\not\subseteq A+A)\nonumber\\
&\ \le \ \sum_{i=\ell}^{\mathclap{2n-\ell-2}}\mathbb{P}(i\notin A+A)\nonumber\\
&\ =\ \sum_{i=\ell}^{n-1}\mathbb{P}(i\notin A+A)+\sum_{i=n}^{\mathclap{2n-\ell-2}}\mathbb{P}(i\notin A+A)\nonumber\\
&\ =\ \sum_{\substack{\ell \leq i<n\\i\text{ \rm odd}}}(2q-q^2)^{(i+1)/2}+\sum_{\substack{\ell \leq i<n\\i\text{ \rm even}}}q(2q-q^2)^{i/2}\nonumber\\
&\ \ \ \ \ \ +\ \sum_{\mathclap{\substack{n\leq i\leq 2n-\ell-2\\i\text{ \rm odd}}}}(2q-q^2)^{n-(i+1)/2}+\sum_{\mathclap{\substack{n\leq i\leq 2n-\ell-2\\i\text{ \rm even}}}}q(2q-q^2)^{n-1-i/2}.
\end{align}
The last equality uses Lemma ~\ref{lem:lemma7MO}. Each of the four sums on the RHS of inequality~\eqref{firstCineq} can be bounded from above  by an infinite geometric sum as follows:
\begin{alignat}{3}
\sum_{\substack{\ell \leq i<n\\i\text{ \rm odd}}}(2q-q^2)^{(i+1)/2}
&\ \le \ \sum_{\substack{i\geq\ell\\i\text{ \rm odd}}}(2q-q^2)^{(i+1)/2}
&\ = \
\begin{cases}
\frac{(2q-q^2)^{j+1}}{(1-q)^2} &\text{ if }\ell = 2j+1,\\
\frac{(2q-q^2)^{j+1}}{(1-q)^2} &\text{ if } \ell = 2j,
\end{cases}\nonumber\\
\sum_{\substack{\ell \leq i<n\\i\text{ \rm even}}}q(2q-q^2)^{i/2}
&\ \le \ \sum_{\substack{i\geq\ell\\i\text{ \rm even}}}q(2q-q^2)^{i/2}
&\ = \
\begin{cases}
\frac{q(2q-q^2)^{j+1}}{(1-q)^2} &\text{ if }\ell = 2j+1,\\
\frac{q(2q-q^2)^j}{(1-q)^2} &\text{ if }\ell = 2j,
\end{cases}\nonumber\\
\sum_{\mathclap{\substack{n \leq i\leq 2n-\ell-2\\i\text{ \rm odd}}}}(2q-q^2)^{n-(i+1)/2}
&\ \le \ \sum_{\mathclap{\substack{i\leq 2n-\ell-2\\i\text{ \rm odd}}}}(2q-q^2)^{n-(i+1)/2}
&\ = \
\begin{cases}
\frac{(2q-q^2)^{j+1}}{(1-q)^2} &\text{ if }\ell = 2j+1,\\
\frac{(2q-q^2)^{j+1}}{(1-q)^2} &\text{ if }\ell = 2j,
\end{cases}\nonumber\\
\sum_{\mathclap{\substack{n \leq i\leq 2n-\ell-2\\i\text{ \rm even}}}}q(2q-q^2)^{n-1-i/2}
&\ \le \ \sum_{\mathclap{\substack{i\leq 2n-\ell-2\\i\text{ \rm even}}}}q(2q-q^2)^{n-1-i/2}
&\ = \
\begin{cases}
\frac{q(2q-q^2)^{j+1}}{(1-q)^2} &\text{ if }\ell = 2j+1,\\
\frac{q(2q-q^2)^j}{(1-q)^2} &\text{ if }\ell = 2j.
\end{cases}
\end{alignat}
Adding these together and seeing that the larger of the two cases obtained is when $\ell=2j$ is even, we obtain the desired bound (inequality~\eqref{eqn:eventCbound}).
\end{proof}

\begin{rek}
In the first step of the above proof, we could replace
\begin{equation}
    \sum_{i=\ell}^{\mathclap{2n-\ell-2}}\mathbb{P}(i\notin A+A)
\end{equation}
with
\begin{equation}
    \sum\limits_{\mathclap{\substack{\ell \leq i \leq 2n-\ell-2 \\ i\equiv l\mod 2}}}\mathbb{P}(i,i+1\notin A+A),
\end{equation}
and then use the results of Proposition ~\ref{prop:p(i,j)} to place a tighter upper bound. However, these terms are already quite small and will play little role in our upper bound, so this would not significantly improve our result.
\end{rek}

Using these lemmas, we prove an upper bound on the probability of missing exactly $k$ elements.

\begin{thm}\label{upperbound}
Pick a fringe width $\ell$. For any $k\in [0,\ell]$,
\be\label{upperboundineq}
    \begin{split}
        m_{n\,;\,p}(k)\ \le \ \sum_{i=0}^{k}\mathbb{P}(L_{i})\mathbb{P}(L_{k-i})+2\frac{(3q-q^2)(2q-q^2)^{\ell/2}}{(1-q)^2}.
    \end{split}
\ee
\end{thm}
\begin{proof}
We divide the interval $[0,2n-2]$ into three subintervals: $[0,\ell-1],[\ell,2n-\ell-2]$ and $[2n-\ell-1,2n-2]$. 
Suppose that there are $k$ missing sums. We separate into two cases.
\mbox{}\newline

\textbf{Case I.} There are no missing sums in the interval $[\ell,2n-\ell-2]$. In this case, let $i$ be the number of missing sums in $[0,\ell-1]$. (Note that $i$ can be any integer between $0$ and $k$ inclusive, because we chose $k\le \ell$.) Then the remaining $k-i$ sums are in $[2n-\ell-2,2n-2]$, and thus the events $L_i$ and $R_{k-i}$ both occur.
\mbox{}\newline

\textbf{Case II.} There is at least one missing sum in $[\ell,2n-\ell-2]$. This corresponds to the event $C$ defined in Lemma ~\ref{lem:eventCbound}.

The above casework gives us the expression
\be \label{caseskmissingsums}
m_{n\,;\,p}(k)\ =\ \sum_{i=0}^k\mathbb{P}(L_i\text{ \rm and }R_{k-i}) + \mathbb{P}(C).
\ee

By Lemma~\ref{fringeindep}, $L_i$ and $R_{k-i}$ are independent, so
\be
\mathbb{P}(L_i\text{ \rm and } R_{k-i}) \ = \ \mathbb{P}(L_i)\mathbb{P}(R_{k-i}) \ = \ \mathbb{P}(L_i)\mathbb{P}(L_{k-i}).
\ee

Using this in ~(\ref{caseskmissingsums}), along with the bound on $\mathbb{P}(C)$ from Lemma~\ref{lem:eventCbound}, gives our desired bound (Inequality~\eqref{upperboundineq}), completing the proof.
\end{proof}
\begin{cor}\label{lowboundf}
Let $k\in\mathbb{N}_0$ and $p\in (0,1)$ be chosen. Given $\ell\ge k$, then \be m_p(k)\ \le\ \sum_{i=0}^{k}\mathbb{P}(L_{i})\mathbb{P}(L_{k-i})+2\frac{(3q-q^2)(2q-q^2)^{\ell/2}}{(1-q)^2}.
\ee
\end{cor}
\begin{proof}
This result follows immediately from Theorem ~\ref{upperbound} by taking the limit as $n$ goes to infinity of both sides. In particular, $\lim_{n\rightarrow\infty}m_{n\,;\,p}(k)=m_p(k)$ while the right side is independent of $n$.
\end{proof}

\subsection{A Lower Bound on $m_p(k)$}\label{lower}

We now attack the more challenging problem of finding a lower bound for the number of missing sums. This will allow us to prove the existence of a divot at $1$ by showing that the probability of missing nothing and the probability of missing two sums have lower bounds that are greater than the upper bound for missing one sum. We begin by outlining how the events $R_k^a$ and $L_k^a$ will assist us. To lower bound $m_p(k)$ it is natural to look for the $k$ missing sums in the two fringes, so we will use the two events $R_{k-i}^a$ and $L_i^a$. Then our key step will be to show that, assuming $R_{k-i}^a$ and $L_i^a$ occur, there are no further elements missing, either from the unspecified elements of the fringes or from the center bulk $[2\ell-1,2n-2\ell-1]$ (these two bounds will be proved quantitatively in Lemmas \ref{epsilon} and \ref{tau}). This is why we introduce notion of $L_k^a$, since $L_k$ may occur with very few elements in $A$ if they are chosen in a clever manner. The event $L_k^a$ will ensure that $A$ contains a more effective number of elements - specifically, $\tau(R_k^a)$. The event $L_k^a$ is diagrammed in Figure \ref{fig:Lak}.

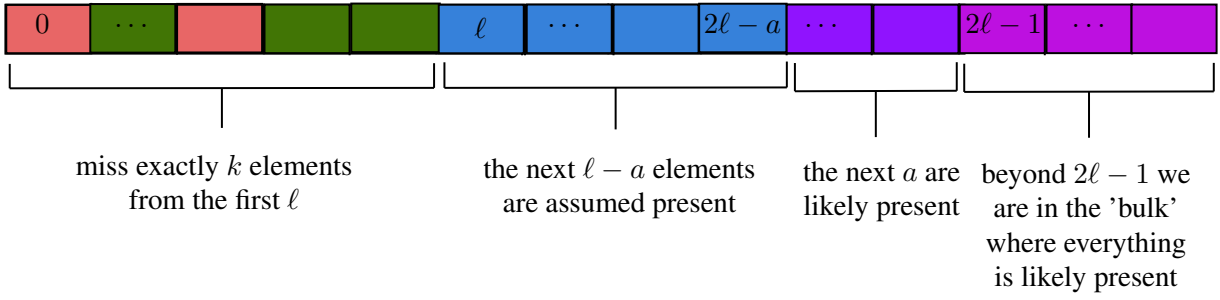
\begin{figure}
    \centering

\tikzset{every picture/.style={line width=0.75pt}} 

\begin{tikzpicture}[x=0.75pt,y=0.75pt,yscale=-1,xscale=0.87]

\draw  [fill={rgb, 255:red, 232; green, 102; blue, 102 }  ,fill opacity=1 ] (100,174.9) -- (149.3,174.9) -- (149.3,199.42) -- (100,199.42) -- cycle ;
\draw  [fill={rgb, 255:red, 232; green, 102; blue, 102 }  ,fill opacity=1 ] (1,175.33) -- (50.3,175.33) -- (50.3,199.86) -- (1,199.86) -- cycle ;
\draw   (0.5,175.33) -- (556.5,175.33) -- (556.5,200) -- (0.5,200) -- cycle ;
\draw    (50.3,174.4) -- (50.3,200) ;
\draw    (100,174.9) -- (100,200.5) ;
\draw    (150.3,174.73) -- (150.3,200.33) ;
\draw    (200,175.23) -- (200,200.83) ;
\draw    (250.3,174.73) -- (250.3,200.33) ;
\draw    (300,175.23) -- (300,200.83) ;
\draw    (350,174.73) -- (350,200.33) ;
\draw    (301.7,175.23) -- (301.7,200.83) ;
\draw    (352.33,174.73) -- (352.33,200.33) ;
\draw    (353.03,175.23) -- (353.03,200.83) ;
\draw    (403.67,174.73) -- (403.67,200.33) ;
\draw    (453.37,175.23) -- (453.37,200.83) ;
\draw    (502.7,175.23) -- (502.7,200.83) ;
\draw    (553.33,174.73) -- (553.33,200.33) ;
\draw    (603.03,175.23) -- (603.03,200.83) ;
\draw    (653.67,174.73) -- (653.67,200.33) ;
\draw  [fill={rgb, 255:red, 65; green, 117; blue, 5 }  ,fill opacity=1 ] (49.7,174.98) -- (99,174.98) -- (99,199.5) -- (49.7,199.5) -- cycle ;
\draw  [fill={rgb, 255:red, 65; green, 117; blue, 5 }  ,fill opacity=1 ] (150.7,175.23) -- (200,175.23) -- (200,199.76) -- (150.7,199.76) -- cycle ;
\draw  [fill={rgb, 255:red, 65; green, 117; blue, 5 }  ,fill opacity=1 ] (201,174.73) -- (252.3,174.73) -- (252.3,199.26) -- (201,199.26) -- cycle ;
\draw  [fill={rgb, 255:red, 54; green, 129; blue, 218 }  ,fill opacity=1 ] (251.7,175.31) -- (301,175.31) -- (301,199.83) -- (251.7,199.83) -- cycle ;
\draw  [fill={rgb, 255:red, 54; green, 129; blue, 218 }  ,fill opacity=1 ] (302.7,175.23) -- (352,175.23) -- (352,199.76) -- (302.7,199.76) -- cycle ;
\draw  [fill={rgb, 255:red, 54; green, 129; blue, 218 }  ,fill opacity=1 ] (353.03,175.23) -- (402.33,175.23) -- (402.33,199.76) -- (353.03,199.76) -- cycle ;
\draw  [fill={rgb, 255:red, 54; green, 129; blue, 218 }  ,fill opacity=1 ] (402.67,174.81) -- (453.37,174.81) -- (453.37,199.33) -- (402.67,199.33) -- cycle ;
\draw  [fill={rgb, 255:red, 144; green, 19; blue, 254 }  ,fill opacity=1 ] (453.4,175.23) -- (502.7,175.23) -- (502.7,199.76) -- (453.4,199.76) -- cycle ;
\draw  [fill={rgb, 255:red, 144; green, 19; blue, 254 }  ,fill opacity=1 ] (503.7,175.31) -- (553,175.31) -- (553,199.83) -- (503.7,199.83) -- cycle ;
\draw  [fill={rgb, 255:red, 189; green, 16; blue, 224 }  ,fill opacity=1 ] (553.73,175.23) -- (603.03,175.23) -- (603.03,199.76) -- (553.73,199.76) -- cycle ;
\draw  [fill={rgb, 255:red, 189; green, 16; blue, 224 }  ,fill opacity=1 ] (604.03,175.23) -- (653.33,175.23) -- (653.33,199.76) -- (604.03,199.76) -- cycle ;
\draw  [fill={rgb, 255:red, 189; green, 16; blue, 224 }  ,fill opacity=1 ] (653.33,175.23) -- (702.63,175.23) -- (702.63,199.76) -- (653.33,199.76) -- cycle ;
\draw    (126,217.4) -- (126,237.5) ;
\draw    (3,217.51) -- (249,217.29) ;
\draw    (3,205.19) -- (3,217.51) ;
\draw    (249,204.83) -- (249,217.29) ;

\draw    (354.25,217.18) -- (354.25,238.25) ;
\draw    (255,217.29) -- (453.5,217.07) ;
\draw    (255,204.38) -- (255,217.29) ;
\draw    (453.5,204) -- (453.5,217.07) ;

\draw    (505.08,218.88) -- (505.08,242.13) ;
\draw    (458.17,219) -- (552,218.75) ;
\draw    (458.17,204.75) -- (458.17,219) ;
\draw    (552,204.33) -- (552,218.75) ;

\draw    (628.67,219.21) -- (628.67,242.46) ;
\draw    (556.33,219.33) -- (701,219.08) ;
\draw    (556.33,205.08) -- (556.33,219.33) ;
\draw    (701,204.67) -- (701,219.08) ;

\draw (16,178.73) node [anchor=north west][inner sep=0.75pt]  [font=\normalsize]  {$0$};
\draw (555.33,178.13) node [anchor=north west][inner sep=0.75pt]  [font=\normalsize]  {$2\ell -1$};
\draw (62,183.73) node [anchor=north west][inner sep=0.75pt]  [font=\normalsize]  {$\dotsc $};
\draw (403.67,178.13) node [anchor=north west][inner sep=0.75pt]  [font=\normalsize]  {$2\ell -a$};
\draw (270.7,180.63) node [anchor=north west][inner sep=0.75pt]  [font=\normalsize]  {$\ell $};
\draw (313,183.73) node [anchor=north west][inner sep=0.75pt]  [font=\normalsize]  {$\dotsc $};
\draw (462,183.73) node [anchor=north west][inner sep=0.75pt]  [font=\normalsize]  {$\dotsc $};
\draw (617,183.73) node [anchor=north west][inner sep=0.75pt]  [font=\normalsize]  {$\dotsc $};
\draw (356.5,268) node   [align=left] {\begin{minipage}[lt]{113.56pt}\setlength\topsep{0pt}
\begin{center}
the next $\displaystyle \ell -a$ elements are assumed present
\end{center}

\end{minipage}};
\draw (121.5,265) node   [align=left] {\begin{minipage}[lt]{115.6pt}\setlength\topsep{0pt}
\begin{center}
miss exactly $\displaystyle k$ elements from the first $\displaystyle \ell $
\end{center}

\end{minipage}};
\draw (508.5,269.5) node   [align=left] {\begin{minipage}[lt]{83.64pt}\setlength\topsep{0pt}
\begin{center}
the next $\displaystyle a$ are likely present
\end{center}

\end{minipage}};
\draw (627,287.88) node   [align=left] {\begin{minipage}[lt]{79.56pt}\setlength\topsep{0pt}
\begin{center}
 beyond $\displaystyle 2\ell -1$ we are in the 'bulk' where everything is likely present
\end{center}

\end{minipage}};

\end{tikzpicture}

    \caption{The event $L_k^a$. Examine the elements of $A+A \subset [0,2n-2]$. We require that $k$ elements are missing (red) from $A+A \cap [0,\ell-1]$ (green), but that the following $\ell-a$ elements are present (blue). We will show below that in the case of the event $L_k^a$ both the remaining $a$ elements below $2\ell-1$ (purple) and bulk above $2\ell-1$ are present with high probability. The same conclusions hold for $R_k^a$ due to symmetry.}
    \label{fig:Lak}
\end{figure}

Once again, we begin by observing that our fringe events are indeed independent.
\begin{lem}[Independence of the Fringes]\label{fringeindep2} Fix a fringe width $\ell$ and a positive integer $a\le \ell$. If $n \ge 4\ell-2a+1$, then for any $k_1,k_2\in [0,\ell]$, the events $L_{k_1}^a$ and $R_{k_2}^a$ are independent.
\end{lem}
The proof of Lemma ~\ref{fringeindep2} is similar to that of Lemma ~\ref{fringeindep}.
The following lemma is a generalization of Proposition 8 in [MO]. 	The lemma gives a lower bound that is independent the specific elements of the fringe. Instead, the bound only involves the cardinalities of $L$ and $R.$

\begin{lem}\label{epsilon}
Choose a fringe width $\ell$ and let $L\subseteq [0,\ell -1]$ and $R\subseteq [n-\ell,n-1]$ be fixed. Let $S=L\cup M\cup R$ for $M\subseteq [\ell, n-\ell-1].$ Then for any $\varepsilon > 0$,
\be\mathbb{P}([2\ell-1,2n-2\ell-1]\subseteq A+A)\ \ge\  1-\frac{1+q}{(1-q)^2}(q^{|L|}+q^{|R|})-\varepsilon\ee
for all sufficiently large $n$.
\end{lem}
\begin{proof}
	We have
	\be
	\begin{split}
		&\mathbb{P}([2\ell-1,2n-2\ell-1]\subseteq A+A)\\&=\ \mathbb{P}([2\ell-1,n-\ell-1]\cup[n+\ell-1,2n-2\ell-1]\subseteq A+A \\&\indent \mbox{ and }[n-\ell,n+\ell-2]\subseteq A+A)\\
		&= \ 1-\mathbb{P}([2\ell-1,n-\ell-1]\cup[n+\ell-1,2n-2\ell-1]\not\subseteq A+A \\&\indent \mbox{ or }[n-\ell,n+\ell-2]\not\subseteq A+A)\\
		&\ge \ 1-\mathbb{P}([2\ell-1,n-\ell-1]\cup[n+\ell-1,2n-2\ell-1]\not\subseteq A+A) \\&\indent -\mathbb{P}([n-\ell,n+\ell-2]\not\subseteq A+A).
	\end{split}
	\ee
	We find a lower bound for $\mathbb{P}([n-\ell,n+\ell-2]\subseteq A+A).$ Since $M+M\subseteq A+A$, \be\begin{split}
		\mathbb{P}([n-\ell,n+\ell-2]\subseteq A+A)\ \ge \ \mathbb{P}([n-\ell,n+\ell-2]\subseteq M+M).
	\end{split}
	\ee
Applying the change of variable $N=n-2\ell$, we estimate
	\begin{align}
		\mathbb{P}([n-2\ell,n-2]\subseteq M+M)\ &=\ \mathbb{P}([N,N+2\ell-2]\subseteq M+M)\nonumber \\
		&=\ 1-\mathbb{P}(\exists k\in [N,N+2\ell-2], k\notin M+M)\nonumber \\ &\ge\ 1-\sum_{k=N}^{N+2\ell-2}\mathbb{P}(k\notin M+M)\nonumber \\
		&=\ 1-\sum_{\mathclap{\substack{N\leq\, k\, \leq N+2\ell-2\\ k\text{ \rm even}}}}\mathbb{P}(k\notin M+M)-\sum_{\mathclap{\substack{N\leq \,k\, \leq N+2\ell-2\\ k\text{ \rm odd}}}}\mathbb{P}(k\notin M+M)\nonumber \\
		&=\ 1- \sum_{\mathclap{\substack{N\leq\, k\, \leq N+2\ell-2\\ k\text{ \rm even}}}}q(2q-q^2)^{N-1-k/2}-\sum_{\mathclap{\substack{N\leq \,k\, \leq N+2\ell-2\\ k\text{ \rm odd}}}}(2q-q^2)^{N-(k+1)/2 }.
	\end{align}
	The last equality uses Lemma ~\ref{lem:lemma7MO}. In the last line, the exponents of $2q-q^2 \in (0,1)$ are all at least $N/2-\ell = n/2 - 2\ell$, so the RHS approaches $1$ as $n\to \infty$. Hence for any $\varepsilon > 0$, when $n$ is sufficiently large we have  $\mathbb{P}([n-\ell,n+\ell-2]\subseteq A+A)\ge 1-\varepsilon$ and so \begin{equation*}
	   \mathbb{P}([n-\ell,n+\ell-2]\not\subseteq A+A) = 1 - \mathbb{P}([n-\ell,n+\ell-2]\subseteq A+A) \le \varepsilon
	\end{equation*}
	
	Combining this with Lemma ~\ref{prop:prop8MO}, we obtain
	\begin{equation}
	\mathbb{P}([2\ell-1,2n-2\ell-1]\subseteq A+A)\ \ge\ 1-\frac{1+q}{(1-q)^2}(q^{|L|}+q^{|R|})-\varepsilon.
	\end{equation}
	This completes our proof.
\end{proof}

The event $L_i^a$ prescribes, in some ways, the behavior of $i+\ell-a$ elements in $[0,2\ell]$; $i$ sums must be missing from the first $\ell$, while $[\ell,2\ell-a]$ are all present. The next lemma places a lower bound on the probability that the remaining $a-3$ elements are also present in $A+A$.
\begin{lem}\label{tau}
	For $n\ge 4\ell -2a +1$, we have
	\be
	\begin{split}
		\mathbb{P}([2\ell-a+1,2\ell-2]&\subseteq A+A\ |\ L^{a}_i)\ \ge\ 1-(a-2)q^{\tau(L^{a}_i)},\\
		\mathbb{P}([2n-2\ell,2n-2\ell+a-3]&\subseteq A+A\ |\ R^{a}_i)\ \ge\ 1-(a-2)q^{\tau(R^{a}_i)}.
	\end{split}
	\ee
\end{lem}
\begin{proof}

We prove only the first inequality because the second follows identically. We have
\begin{align}
\mathbb{P}([2\ell-a+1,2\ell-2]\subseteq A+A\ |\ L^{a}_i)
&\ = \ \ 1-\mathbb{P}([2\ell-a+1,2\ell-2]\not\subseteq A+A\ |\ L^{a}_i)\nonumber \\
&\ \ge\  \ 1-\sum_{\mathclap{k=2\ell-a+1}}^{\mathclap{2\ell-2}}\mathbb{P}(k\notin A+A\ |\ L^{a}_i).
\end{align}
Recall the definitions of $\mathcal{M}_{L,a,i}$ (the set of sets $L\subset [0,\ell-1]$ such that event $L^{a}_i$ occurs) and \be\tau\big(L^{a}_k\big)\ =\ \min_{L\in\mathcal{M}_{L,i}}|L\cap [0,\ell-a+1]|.\ee
Suppose from now on that $L_i^a$ occurs. For each $k\in[2\ell-a+1,2\ell-2]$, the probability that $k\notin A+A$ is equal to the probability that for each $x\in L$, the corresponding $x-k\notin L$. Since there are at least $\tau(L_k^a)$ elements of $L$, and the probability of excluding a certain integer from $S$ is $q$, we can bound
\be
\mathbb{P}(k\notin A+A\ | \ L_i^a)\ \le\ q^{\tau(R_i^a)}.
\ee
Hence
\be\label{taubound}
1-\sum_{\mathclap{k=2\ell-a+1}}^{\mathclap{2\ell-2}}\mathbb{P}(k\notin A+A\ |\ L^{a}_i)\ \ge\ 1-(a-2)q^{\tau(R^{a}_i)}.
\ee
This completes our proof.
\end{proof}
Given $k\in\mathbb{N}_0$, the following theorem gives us a lower bound for $m_{n\,;\,p}(k)$.
\begin{thm}\label{lowerboundn}
	Fix $q\in (0,1)$ and pick a fringe length $\ell\ge 0$. Also choose $a\le \ell$. For any $\varepsilon>0$, the following holds for all sufficiently large $n$:
	\be\begin{split}& m_{n\,;\,p}(k)\ \ge \ \sum_{i=0}^{k}\mathbb{P}(L^{a}_i)\mathbb{P}(R^{a}_{k-i})\theta_{k,i}(q,\varepsilon),\end{split}\ee
where \be\theta_{k,i}(q,\varepsilon)\ = \ 1-(a-2)(q^{\tau(L^{a}_i)}+q^{\tau(R^{a}_{k-i})})-\varepsilon-\frac{1+q}{(1-q)^2}(q^{\min L_i^a}+q^{\min R^{a}_{k-i}}).\ee
\end{thm}

\begin{proof}
	The probability that $A+A$ is missing exactly $k$ sums is greater than the probability that all these sums are missing from the two fringes. Thus for each $k\in [0,\ell]$, we have
	\begin{align}
		m_{n\,;\,p}(k)&\ \ge\ \sum_{i=0}^{k}\mathbb{P}(L^{a}_i\mbox{ and } R^{a}_{k-i}\mbox{ and } [2\ell-a+1,2n-2\ell+a-3]\subseteq A+A)\nonumber\\
		&\ = \ \ \mathbb{P}(L^{a}_i\mbox{ and }R^{a}_i)\mathbb{P}([2\ell-a+1,2n-2\ell+a-3]\subseteq A+A\ |\ L^{a}_i\mbox{ and } R^{a}_{k-i})\nonumber\\
		&\ = \ \ \mathbb{P}(L^{a}_i)\mathbb{P}(R^{a}_i)\mathbb{P}([2\ell-a+1,2n-2\ell+a-3]\subseteq A+A\ |\ L^{a}_i\mbox{ and } R^{a}_{k-i}).
	\end{align}
	This last equality follows from Lemma ~\ref{fringeindep2}. We can bound $\mathbb{P}([2\ell-a+1,2n-2\ell+a-3]\subseteq A+A\ |\ L^{a}_i\mbox{ and } R^{a}_{k-i})$ below by splitting into three subintervals. 
	\begin{multline*}
		\mathbb{P}([2\ell-a+1,2\ell-2]\cup[2\ell-1,2n-2\ell-1]\cup[2n-2\ell, 2n-2\ell+a-3]\subseteq A+A\ |\ L^{a}_i\mbox{ and } R^{a}_{k-i})\\
		\begin{aligned}
		&\begin{multlined}
		=\ 1-\mathbb{P}([2\ell-a+1,2\ell-2]\not\subseteq A+A\mbox{ or }[2\ell-1,2n-2\ell-1]\not\subseteq A+A\\\mbox{ or }[2n-2\ell,2n-2\ell+a-3]\not\subseteq A+A\ |\ L^{a}_i\mbox{ and } R^{a}_{k-i})
		\end{multlined}\\
		&\begin{multlined}
		\ge \ 1-\mathbb{P}([2\ell-a+1,2\ell-2]\not\subseteq A+A\ |\ L^{a}_i\mbox{ and } R^{a}_{k-i})\\-\mathbb{P}([2\ell-1,2n-2\ell-1]\not\subseteq A+A\ |\ L^{a}_i\mbox{ and } R^{a}_{k-i})\\-\mathbb{P}([2n-2\ell,2n-2\ell+a-3]\not\subseteq A+A\ |\ L^{a}_i\mbox{ and } R^{a}_{k-i})
		\end{multlined}\\
		&\ge \ 1-(a-2)q^{\tau(L^{a}_i)}-\frac{1+q}{(1-q)^2}(q^{\min L_i^a}+q^{\min R^{a}_{k-i}})-\varepsilon-(a-2)q^{\tau(R^{a}_{k-i})}.\\
		\end{aligned}
	\end{multline*}
	The last inequality uses Lemma ~\ref{tau}, as well as Lemma ~\ref{epsilon} with the observation that for any $L$ such that $L_i^a$ occurs, $q^{|L|}\le q^{\min L_i^a}$ (respectively $q^{|R|} \le q^{\min R_i^a}$).
	Hence
	\begin{multline}
	\mathbb{P}(L^{a}_i\mbox{ and } R^{a}_{k-i}\mbox{ and } [2\ell-a+1,2n-2\ell+a-3]\subseteq A+A)
	\ge \ \mathbb{P}(L^{a}_i)\mathbb{P}(R^{a}_{k-i})\theta_{k,i}(q,\varepsilon).
	\end{multline}
	This completes our proof.
\end{proof}

\begin{cor}\label{upboundf}
Let $k\in\mathbb{N}_0$ and $p\in (0,1)$ be chosen. Given $\ell\ge k$, then \begin{align}
m_p(k)\ \ge \ \sum_{i=0}^{k} \mathbb{P}(L^{a}_i)\mathbb{P}(L^{a}_{k-i}) \bigg[1&-(a-2)\left(q^{\tau(L^{a}_i)}+q^{\tau(L^{a}_{k-i})}\right)\nonumber \\
 &-  \frac{1+q}{(1-q)^2}\left(q^{\min L_i^a}+q^{\min L^{a}_{k-i}}\right)\bigg].
\end{align}
\end{cor}

\begin{proof}
This follows immediately from Theorem ~\ref{lowerboundn} by taking the limit as $n$ goes to infinity of both sides. In particular, $\lim_{n\rightarrow\infty}m_{n\,;\,p}(k)=m_p(k)$ while the right side is independent of $n$.
\end{proof}

\subsection{Computational Results}\label{ssec:computation}
Thus far we have seen lower (Corollary \ref{upboundf}) and upper (Corollary \ref{lowboundf}) bounds on $m_p(k)$. Fixing the integer $k$, these two corollaries bounding $m_p(k)$ depend on
\begin{itemize}
    \item the value $p \in (0,1)$,
    \item the fringe with $\ell \geq k$,
    \item the probabilities $\mathbb{P}(L_i)$ for $0\leq i \leq k$,
    \item the integer $a  < \ell$ discussed above,
    \item the probabilities $\mathbb{P}(L_i^a)$ for $0\leq i \leq k$,
    \item and the minima $\tau(L_i^a)$ and $\min L_i^a$ for $0\leq i \leq k$.
\end{itemize}
These upper and lower bounds enable a proof of Theorem ~\ref{thm:divot}. After fixing $a$, the probabilities $\mathbb{P}(L_i)$ and $\mathbb{P}(L_i^a)$ are (complicated) polynomial functions of $p$, since they only check the fringe elements of $A$. Similarly, the minima $\tau(L_i^a)$ and $\min L_i^a$ are, given $a$, fixed integers that a computer can determine in reasonable time. We formalize this observation in the remark below.

\begin{rek}\label{knowall}
Given $0\le k\le \ell$, $\mathbb{P}(L_k)$ is a polynomial of $p$. Recall that $\mathcal{M}_{L,k}$ is the set of all sets $L$ such that event $L_k$ occurs. For $0\le i\le \ell$, let \be
c(i)\ = \ |\{L\in\mathcal{M}_{L,k}\ \mbox{ such that }\ |L| = i\}|.
\ee
Then \be \mathbb{P}(L_k) \ = \ \sum_{i=0}^{\ell}c(i)p^{i}(1-p)^{\ell-i}.\ee
The coefficients $c(i)$ can be numerically computed. Similarly, define
\be c_{k,a}(i)\ = \ |\{L\in\mathcal{M}_{L,a,k}\ \mbox{ such that }\ |L| = i\}|.\ee
Then
\be \mathbb{P}(L^a_k) = \sum_{i=0}^{\ell}c_{k,a}(i)p^i(1-p)^{\ell-i}.\ee
So we can numerically compute the upper and lower bounds for $m_p(k)$ found in \textsection~\ref{upper} and \textsection~\ref{lower}.
\end{rek}

Finally, we employ these numerical techniques through exhaustive search to prove Theorem ~\ref{thm:divot}.

\begin{proof}[Proof of Theorem ~\ref{thm:divot}]
For our argument for the divot at $1$, we use $\ell=30$ and $a=12.$ These numbers are chosen in two stages. First $\ell=30$ is appropriate because checking all $2^{30} \approx 10^{9}$ possible choices of fringe was near the boundary of what our available computational resources could accomplish. Second $a=12$ was chosen by experimentally investigating which choice of $a$ gave an effective bound. For clarity,
we summarize these bounds:
\begin{align}
m_p(0)\ & \ \ge\ \ \mbox{LB}(0,p),\nonumber\\
m_p(1)\ &\ \le\ \mbox{UB}(1,p),\nonumber\\
m_p(2)\ &\ \ge\ \mbox{LB}(2,p),
\end{align}
where
\begin{align}
    \mbox{LB}(0,p)& \ \vcentcolon = \
    \mathbb{P}(L^{12}_i)\mathbb{P}(L^{12}_{0})\bigg[1-10(q^{\tau(L^{12}_0)}+q^{\tau(L^{12}_{0})})-\frac{1+q}{(1-q)^2}(q^{\min L_0^{12}}+q^{\min L^{12}_{0}})\bigg],\nonumber\\
    \mbox{UB}(1,p)& \ \vcentcolon = \
    \sum_{i=0}^{1}\mathbb{P}(L_{i})\mathbb{P}(L_{1-i})+2
	\frac{(3q-q^2)(2q-q^2)^{15}}{(1-q)^2},\nonumber\\
	\mbox{LB}(2,p)& \ \vcentcolon = \
	\mathbb{P}(L^{12}_i)\mathbb{P}(L^{12}_{2-i})\sum_{i=0}^{2}\bigg[1-10(q^{\tau(L^{12}_i)}+q^{\tau(L^{12}_{2-i})}) \nonumber\\
	&\qquad\qquad\qquad\qquad\qquad\quad-\frac{1+q}{(1-q)^2}(q^{\min L_i^{12}}+q^{\min L^{12}_{2-i}})\bigg].
\end{align}

Using Remark~\ref{knowall}, we can plot each function $\mbox{LB}(0,p),\mbox{UB}(1,p)$ and $\mbox{LB}(2,p)\mbox{ }(q=1-p)$. We provide the values for $\min L^{12}_{i}, \tau(L^{12}_i), c_k(i)$ and $c_{k,a}(i)$, which are crucial values for explicitly plotting
functions $\mbox{LB}(0,p), \mbox{UP}(1,p), \mbox{LB}(2,p)$ in Appendix ~\ref{Data}. Figure ~\ref{divot1} is the plot. From Figure ~\ref{divot1}  and using computer algebra software to check that the ordering suggested by the plot is correct, we see that for $p\ge 0.68$, $\mbox{LB}(0,p)>\mbox{UB}(1,p)<\mbox{LB}(2,p);$ thus, there is a divot at $1$. Numerical evidence shows that our upper bounds are very good; we also discuss this in Appendix ~\ref{whysharpbound}.
\begin{figure}[ht!]
\centering
\includegraphics[scale=.7]{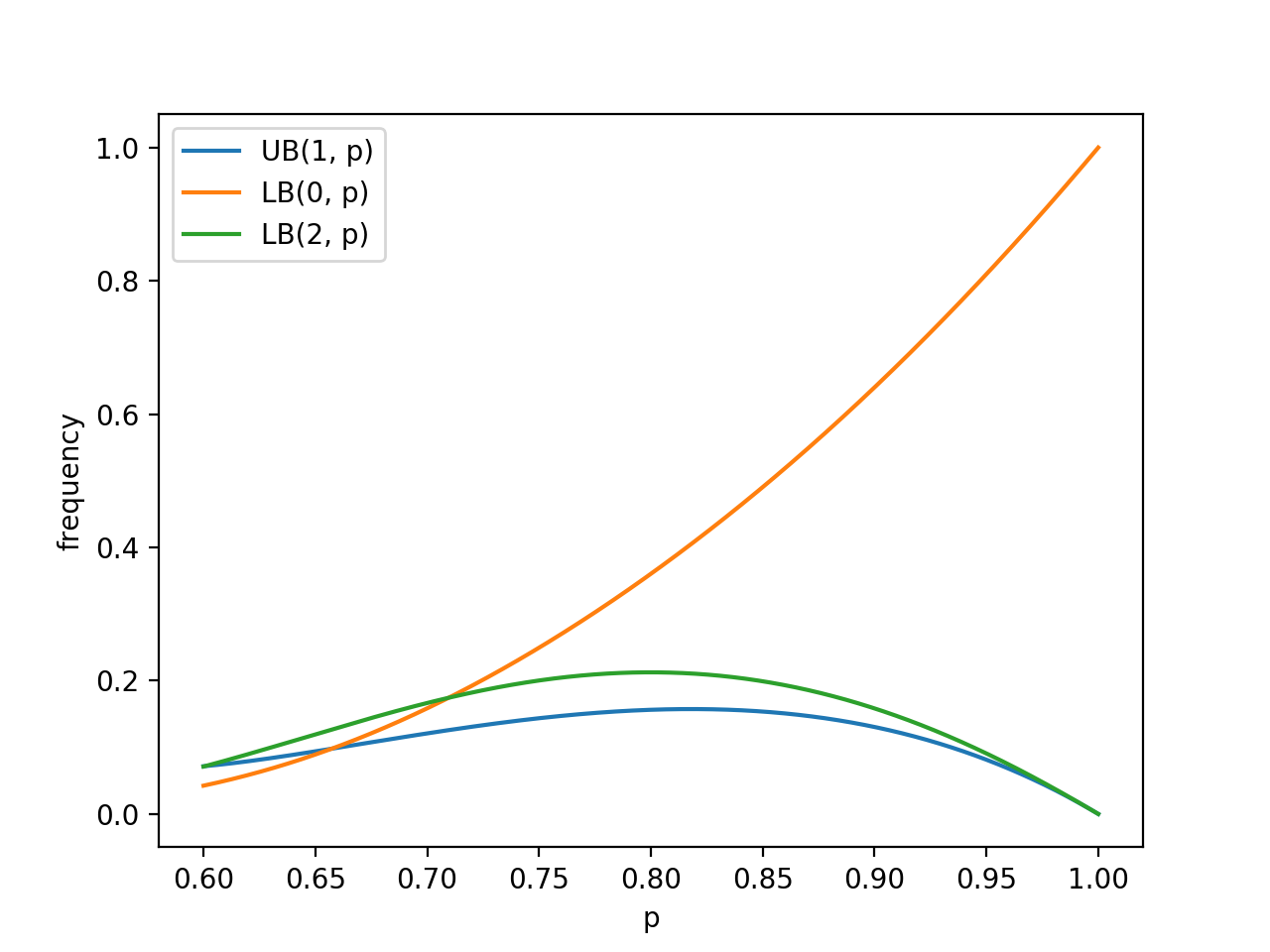}
\caption{Plot of the bounds $\mbox{LB}(0,p), \mbox{UB}(1,p) \mbox{ and }\mbox{LB}(2,p)$. Since $\mbox{LB}(0,p)>\mbox{UB}(1,p)<\mbox{LB}(2,p)$ for $p\ge 0.68$, there exists a divot at $1$ for $p\ge 0.68$.}
\label{divot1}
\end{figure}
\end{proof}

We end this section by discussing some details of our computational process. We can perform calculations for any value of $p$ simultaneously by doing one pre-computation. The program (exhaustively and naively) lists all the fringe sets of a given size at once; the probability we choose each set of a given size can be computed easily. In particular, let $\ell$ be the fringe size (i.e. the width of the fringe) and let $k$ be the number of elements in our set. The probability we choose this set is $p^k(1-p)^{\ell-k}$. We then compute how many missing sums this fringe generates and update this information in a list; exactly those $c(i)$ detailed in Remark \ref{knowall}. We compute $c(i)$ the number of fringes of size $i$ which miss $k$ elements in the sumset, and since these sets all appear with equal probability (no matter the $p$), this scheme does not depend on the choice of $p$.

We picked a fringe size of $\ell = 30$ to compute all the data necessary for Theorem \ref{thm:divot} about the divot at $1$ (see Section \S\ref{sec:divot} and Appendix \ref{Data} for details). The process took three days on a single machine without parallelization. We were initially hopeful to find some results about the divot at $3$ with the same fringe analysis by using a larger fringe size $\ell=40$ and parallelization. Unfortunately, this was too big a search space. We were able to use the shared Linux computing cluster at Williams College for six months, which was enough time to do three-fourths of the computation, but the effort ended due to a memory error. As the run-time for difference set computations is on the order of the square of the same for sumsets (as the two fringes interact), these run-times illustrate the challenges in numerically exploring difference sets.

\section{Correlated Sumsets}\label{sec:correlated}

Up unto this point, we have studied the random variable $|A+A|$, where each element is included in $A$ with probability $p$. Now, we examine the random variable $|A+B|$, where, for a given triplet $(p,p_1,p_2)$ and any $i \in \{0,\dots,n-1\}$

\begin{itemize}
    \item   $\mathbb{P}(i \in A) = p$
    \item $\mathbb{P}(i \in B\ | \ i \in A) = p_1$
    \item $\mathbb{P}(i \in B\ | \ i \not \in A) = p_2$.
\end{itemize}

For example, if $p_1=1,p_2=0$ we recover the problem of $|A+A|$, while if $p_1=0,p_2=1$ we get $|A+A^c|$.

Our first objective is, as before, to use graph theory to compute $\mathbb{P}(i,j \not \in A+B)$. The probability of missing a single element, $\mathbb{P}(i \not \in A+B)$, was computed in \cite{DKMMW}. The clear choice of graph-theoretic generalization is to form a bipartite graph $CG$.

\begin{definition}\label{def:correlatedconditiongraph}
For sets $V=A\cup B=\{0_A,1_A,\dots,  (n-1)_A,0_B,1_B,\dots,  (n-1)_B \}$ and $F \subseteq  [0,2n-2]$ we define the bipartite \emph{correlated condition graph} $CG_F=(V,E)$ induced on $V$ by $F$ where for two vertices $k_1\in A$ and $k_2\in B$, $(k_1,k_2) \in E$ if $k_1+k_2 \in F$. For notational convenience, if $F = \{i,j\}$, we denote $CG_F$ by $CG_{i,j}$.
\end{definition}

Then, just as before with Lemma ~\ref{lem:vertexcover}, the event $i,j \not \in A+B$ is the same as having an independent set on this graph of those elements from $A$ and $B$. Fortunately, the structure of this graph is entirely analogous to that found in \S\ref{sec:graph}. If $k \in \{0,\dots,n-1\}$, and we denote by $k_A$ and $k_B$ the copies of $k$ potentially present in $A,B$ respectively, then we know that if $k_{1,A}+k_{2,B}=i$, then also $k_{1,B}+k_{2,A}=i$, and so each edge in our correlated condition graph has a ``partner''. Thus, \cite{LMO}'s Proposition 3.1 still applies and we once again find ourselves with a collection of disjoint paths, present in pairs where one element is in $A$ and the other is in $B$.

\begin{definition}\label{def:accordionpath}
Let $V=A\cup B=\{0_A,1_A,\dots,  (m-1)_A,0_B,1_B,\dots,  (m-1)_B \}$ and $i,j$ be non-negative integers with $i,j \leq m-1$. Then, an \emph{accordion path} of length $n$ on $CG_{i,j}$ is a pair of paths in $CG_{i,j}$ given by vertices specified by a sequence of integers $k_s$ for $1 \leq s \leq n$, so that for each $s>1$, $k_s+k_{s-1}\in \{i,j\}$. Then the accordion path is given by $\bigcup\limits_{k_A,k_B\,\in\, k_s}\{k_A,k_B \}$ and the edges (inherited from the condition graph) between them.
\end{definition}

\begin{figure}[ht]
   \begin{tikzpicture}[scale=.7,auto=left,every node/.style={circle,fill=black!20,minimum size=12pt,inner sep=2pt}] 

 \node (n0) at (0,0)  {$k_{1,A}$};
 \node (n1) at (1.5,0)  {$k_{2,A}$};
 \node (n2) at (3,0)  {$k_{3,A}$};
 \node (n3) at (4.5,0)  {$\dots$};
 \node (n4) at (0,1.5)  {$k_{1,B}$};
 \node (n5) at (1.5,1.5)  {$k_{2,B}$};
 \node (n6) at (3,1.5)  {$k_{3,B}$};
 \node (n7) at (4.5,1.5)  {$\dots$};
 \node (n8) at (6.0,1.5)  {$k_{n,B}$};
 \node (n9) at (6,0)  {$k_{n,A}$};

   \foreach \from/\to in {n1/n6,n5/n2,n0/n5,n1/n4,n2/n7,n3/n6,n8/n3,n9/n7}
    \draw (\from) -- (\to);

 \end{tikzpicture}
 \caption{A generic accordion path.}
\label{fig:correlated_condition_example}
\end{figure}
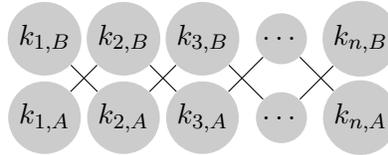

An example of such an accordion path is given in Figure \ref{fig:correlated_condition_example}. In Section ~\ref{sec:graph}, we were able to compute the probability of finding an independent set on a path using a one-dimensional recurrence relation. Here, we still have that vertices in distinct pairs of paths are independent, but within a pair of paths composing a single accordion we have serious dependencies, since each element is included in $B$ with probability conditioned on whether or not that element was included in $A$. However, by setting $S=A\cup B$ and extending our recurrence relation to include two new variables as follows
\begin{align}
    x_n&\ = \ \mathbb{P}(S \text{ is an independent set}),\nonumber\\
    y_n&\ = \ \mathbb{P}(S \text{ is an independent set and }k_n \not\in A),\nonumber\\
    z_n&\ = \ \mathbb{P}(S \text{ is an independent set and }k_n \not\in B),
\end{align}
we then get the recurrence relations
\begin{align}
    x_n&\ = \ qq_2x_{n-1}+qp_2y_{n-1}+pq_1z_{n-1}+pp_1qq_2x_{n-2},\nonumber\\
    y_n&\ = \ qq_2x_{n-1}+qp_2y_{n-1},\nonumber\\
    z_n&\ = \ qq_2x_{n-1}+pq_1z_{n-1},
\end{align}
where $x_1=y_1=z_1=1$ and
\begin{align}
    x_2&\ = \ q(q_2+p_2q)+p(q_1qq_2+q_1^2p+p_1qq_2),\nonumber\\
    y_2&\ = \ q(1-pp_2),\nonumber\\
    z_2&\ = \ pq_1(qq_2+pq_1)+qq_2.\label{eqn:initialvalues}
\end{align}

This larger recurrence relation generalizes the one previously derived in \S\ref{sec:graph}. To find asymptotics, we can examine the eigenvalues of the governing $4 \times 4$ matrix;
\begin{equation}\label{eqn:governingmatrix}
\begin{pmatrix}
x_{n+1} \\
y_{n+1}\\
z_{n+1}\\
x_{n}\\
\end{pmatrix}
\ = \
\begin{pmatrix}
qq_2 & qp_2 & pq_1 & pp_1qq_2 \\
qq_2 & qp_2 & 0 & 0 \\
qq_2 & 0 & pq_1 & 0 \\
1 & 0 & 0 & 0
\end{pmatrix}
\begin{pmatrix}
x_{n} \\
y_{n}\\
z_{n}\\
x_{n-1}\\
\end{pmatrix}.
\end{equation}

In fact, we can find the closed form for the eigenvalues of the governing matrix from Equation \eqref{eqn:governingmatrix}; however, we do not include it in this paper as it is not very informative to the behavior of the probability of obtaining an independent set. Instead, we propose that one might fix one or more of $p$, $p_1$, and $p_2$, and then find the eigenvalues, to obtain a more meaningful result.

From this, we are able to find a preliminary result for the event $k \not\in A+B$, noting the similarities to cases discussed in Section ~\ref{sec:expecvalue} and \cite{MO}.

\begin{proposition}
For $k \in [0,2n-2]$, we have
\[
\mathbb{P}(k \not \in A+B)\ = \ \begin{cases}x_2^{k/2}(1-pp_1) & \text{{\rm if} } k \text{ {\rm is\ even}}, \\ x_2^{(k+1)/2} &  \text{{\rm if} } k \text{ {\rm is\ odd}},\end{cases}
\]
where $x_2$ is as defined in Equation ~\ref{eqn:initialvalues}.
\end{proposition}

\begin{proof}
Consider $CG_2$ induced on $k\in [0,2n-2]$. We notice this graph is very similar to the graph displayed in Figure ~\ref{fig:condition_simple_example}, with disjoint edges and isolated vertices. We also note that Lemma ~\ref{lem:vertexcover} still applies, so we find an independent set on this graph.

By definition, the probability of obtaining an independent set on a disjoint edge is $x_2$. We must count how many of these disjoint edges there are; then we can multiply these together and find the probability of obtaining an independent set on the graph.

If $k$ is odd, then there are $(k+1)/2$ disjoint edges. So, we get $x_2^{(k+1)/2}$.

If $k$ is even, then there are $k/2$ disjoint edges, however, there is also an edge between $(k/2)_A$ and $(k/2)_B$ with no ``partner''. The probability of obtaining an independent set for this edge is $1-pp_1$. So, we get $x_2^{k/2}(1-pp_1)$.
\end{proof}

And using the framework developed in Section ~\ref{sec:graph}, we are able to find the following Proposition.

\begin{proposition}
For $i,j \in [0,2n-2]$, we have
\begin{equation}
\mathbb{P}(i,j \not \in A+B)\ = \ \begin{cases}x_q^s\, x_{q+2}^{s'} & i,j \text{ {\rm both\ odd}, }\\ (1-pp_1)\,x_o\, x_q^s \,x_{q+2}^{s'} & i \text{ {\rm even}, } j \text{ {\rm odd}, }\\(1-pp_1)\, x_{o'}\,x_q^s\, x_{q+2}^{s'} & i \text{ {\rm odd}, } j \text{ {\rm even}, }\\(1-pp_1)^2\, x_o \,x_{o'}\, x_q^s\, x_{q+2}^{s'} & i,j \text{ {\rm both\ even}, }\end{cases}
\end{equation}
where $q,s,s',o,o'$ are as defined in Proposition ~\ref{prop:p(i,j)}.
\end{proposition}
\begin{proof}

Consider $CG_{i,j}$ induced on $[0,n-1]$. To find $\mathbb{P}(i,j\not\in A+B)$, we must find the probability of obtaining an independent set on this graph. Thankfully, the structure of this graph has been well-studied, from Proposition ~\ref{prop:p(i,j)}. The difference is we now have accordion paths as opposed to paths, however $x_n$ gives us the probability of obtaining an independent set on an accordion path of length $n$. So, we can use Proposition ~\ref{prop:p(i,j)} to find the number and lengths of these accordion paths, to obtain our desired result.
\end{proof}

\section{Future Work}\label{sec:future}

We list some natural questions for future research. We first list questions relating to Sections ~\ref{sec:expecvalue}, ~\ref{sec:variance} and ~\ref{sec:correlated}.

\begin{itemize}
    \item Does there exist a ``good'' lower bound for $\mathbb{E}[|A+A|]$ for $p \leq 1/2$?
    \item Can a ``good'' bound be found for $\Var(|A+A|)$?
    \item Does there exist a closed formula for $\mathbb{E}[|A+B|]$ and $\Var(|A+B|)$?
\end{itemize}

Now we list questions relating to Section ~\ref{sec:divot}. For convenience, we present Figure ~\ref{fig: intro} again.

\begin{figure}[ht!]
\centering
\includegraphics[scale=.8]{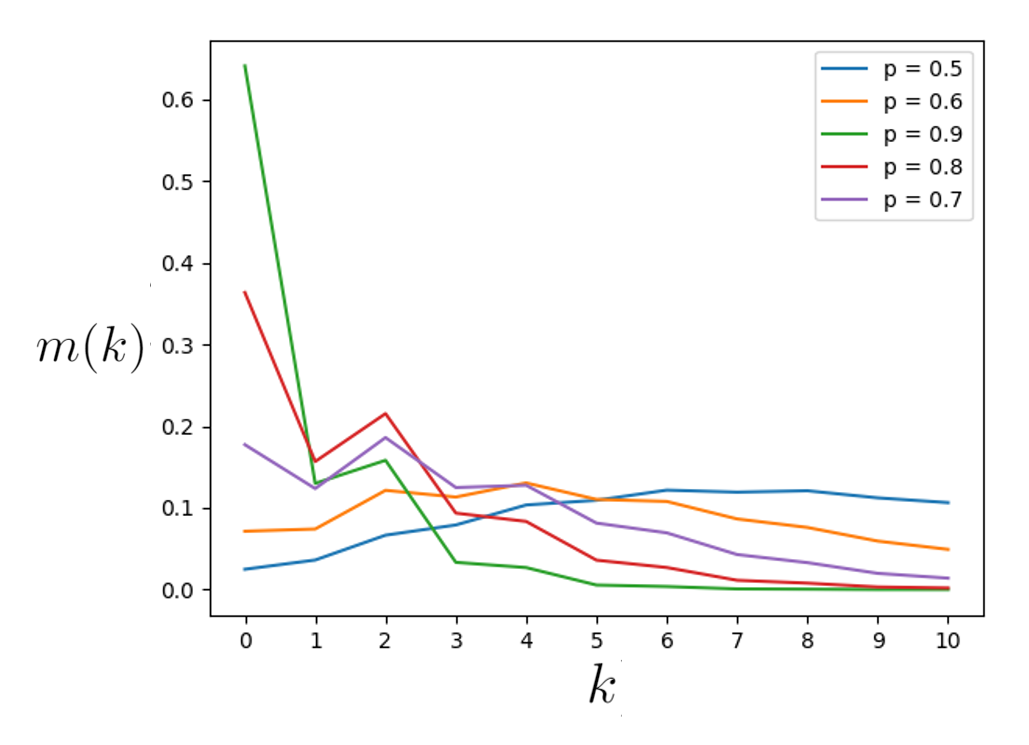}
\caption{Plot of the distribution of missing sums, varying $p$ by simulating $10^6$ subsets of $\{0,1,2,\dots,400\}$. The simulation shows that: for $p=0.9$ and $0.8$, there is a divot at $1$, for $p=0.7$, there are divots at $1$ and $3$, for $p=0.6,$ there is a divot at $3$ and for $p=0.5$, there is a divot at $7$.}
\label{introrepeat}
\end{figure}

\begin{itemize}
    \item As $p$ decreases, the divot appears to shift to the right, from $1$ at $p = .8$, to $3$ for $p = .6$, to $7$ for $p = .5$. How does the position of the divot depend on $p$? Do divots move monotonically with $p$?
    \item At $p = .7$ there appear to be two divots at $1$ and $3$; for what values of $p$ are there more than one divot?
    \item Is there a value $p_0$ where for $p>p_0$ the distribution of the number of missing sums has a divot, and for $p< p_0$ the divot disappears. Where is this phase transition point $p_0$?
    \item In our theoretical and numerical investigations, we have never seen a divot at an even number. Are there no divots at even values?
\end{itemize}

These results all apply to the sumset $A+A$. In general, can any of these results be applied to the difference set $A-A$? What is needed to apply these results to the difference set?


\appendix

\section{Proofs of Generalizations}\label{app:proofs}

Here we provide the full proofs of the many generalizations of lemmas originally proven by \cite{MO}. Note that we have only introduced new notation that generalizes the previous arguments made. We also provide the full proof of Theorem ~\ref{thm:m of k}, that is structurally equivalent to Theorem 1.2 of \cite{LMO}.

\begin{proof}[Proof of Lemma ~\ref{lem:lemma5MO}]
Define random variables $X_j$ by setting $X_j=1$ if $j\in A$ and $X_j=0$ otherwise. By the definition of $A$, the variables $X_j$ are independent random variables for $\ell\le j\le n-u-1$, each taking the values 0 and 1 with probability $q$ and $p$ respectively, while the variables $X_j$ for $0\le j\le\ell-1$ and $n-u\le j\le n-1$ have values that are fixed by the choices of $L$ and $U$.

We have $k\notin A+A$ if and only if $X_j X_{k-j} = 0$ for all $0 \le j \le k/2$; the key point is that these variables $X_j X_{k-j}$ are independent of one another. Therefore
    \begin{equation}
    \Prob{k\notin A+A}\ =\ \prod_{0\le j\le k/2} \Prob{X_j X_{k-j} = 0}.
    \end{equation}
If $k$ is odd, this becomes
    \begin{align}
    \Prob{k\notin A+A} &\ =\ \prod_{j=0}^{\ell-1} \Prob{X_j X_{k-j} = 0} \prod_{j=\ell}^{(k-1)/2} \Prob{X_j X_{k-j} = 0} \nonumber\\
    &\ =\ \prod_{j\in L} \Prob{X_{k-j} = 0} \prod_{j=\ell}^{(k-1)/2} \Prob{X_j = 0 \text{ or } X_{k-j} = 0} \nonumber\\
    &\ =\ q^{|L|} (1-p^2)^{(k+1)/2-\ell}.
    \end{align}
On the other hand, if $k$ is even then
    \begin{align}
    \Prob{k\notin A+A} &\ =\  \prod_{j=0}^{\ell-1} \Prob{X_j X_{k-j} = 0} \bigg( \prod_{j=\ell}^{k/2-1} \Prob{X_j X_{k-j} = 0} \bigg) \Prob{X_{k/2}X_{k/2} = 0} \nonumber\\
    &\ = \ \prod_{j\in L} \Prob{X_{k-j} = 0} \bigg( \prod_{j=\ell}^{k/2-1} \Prob{X_j = 0 \text{ or } X_{k-j} = 0} \bigg) \Prob{X_{k/2} = 0} \nonumber\\
    & \ = \ q^{|L|} (1-p^2)^{k/2-\ell} \cdot q.
    \end{align}
\end{proof}

\begin{proof}[Proof of Lemma ~\ref{lem:lemma6MO}]
This follows from Lemma ~\ref{lem:lemma5MO} applied to the parameters $\ell'=u$ and $L'=n-1-U$, $u'=\ell$ and $U'=n-1-L$, and $A'=n-1-A$ and $k'=2n-2-k$.
\end{proof}

\begin{proof}[Proof of Proposition ~\ref{prop:prop8MO}]
We employ the crude inequality
    \begin{eqnarray}
    & &\Prob{\{2\ell-1,\dots,n-u-1\} \cup \{n+\ell-1,\dots,2n-2u-1\} \not\subseteq A+A} \nonumber\\
      & & \ \ \ \ \   \le\ \sum_{k=2\ell-1}^{n-u-1} \Prob{k\notin A+A} + \sum_{k=n+\ell-1}^{2n-2u-1} \Prob{k\notin A+A}.
    \end{eqnarray}
The first sum can be bounded, using Lemma ~\ref{lem:lemma5MO}, by
    \begin{align}
    \sum_{k=2\ell-1}^{n-u-1} \Prob{k\notin A+A} & \ < \ \sum_{\substack{k\ge2\ell-1 \\ k\text{ odd}}} q^{|L|} (1-p^2)^{(k+1)/2-\ell} + \sum_{\substack{k\ge2\ell-1 \\ k\text{ even}}} q^{|L|+1} (1-p^2)^{k/2-\ell} \nonumber\\
    &\ = \  q^{|L|} \sum_{m=0}^\infty (1-p^2)^m + q^{|L|+1} \sum_{m=0}^\infty (1-p^2)^m\nonumber\\
    & \ = \ q^{|L|} \frac{1}{p^2} + q^{|L| + 1} \frac{1}{p^2}\ =\ \frac{1+q}{p^2}\ q^{|L|}.
    \end{align}
The second sum can be bounded in a similar way using Lemma ~\ref{lem:lemma6MO}, yielding
    \begin{equation}
    \sum_{k=n+\ell-1}^{2n-2u-1} \Prob{k\notin A+A} \ <\   \frac{1+q}{p^2}\ q^{|U|}.
    \end{equation}
Therefore $\Prob{\{2\ell-1,\dots,n-u-1\} \cup \{n+\ell-1,\dots,2n-2u-1\} \not\subseteq A+A}$ is bounded above by\\ $\frac{1+q}{p^2}\pn{q^{|L|} + q^{|U|}}$, which is equivalent to the statement of the proposition.
\end{proof}

\begin{proof}[Proof of Theorem ~\ref{thm:m of k}]
For the lower bound, we construct many $A$ such that $A+A$ is missing $k$ elements. First suppose that $k$ is even. Let the first $k/2$ non-negative integers not be in $A$. Then let the rest of the elements of $A$ be any subset $A'$ that fills in (so $A'+A'$ has no missing elements between its largest and smallest elements); that is $M_{n-k/2}(A') = 0$. By Proposition ~\ref{prop:prop8MO}, we can show that $\mathbb{P}(M_{[0, n-1]}(A') = 0)$ is a constant independent of $n$. If $L \subseteq [0, \ell-1]$ and $U\subseteq [n-u, n-1]$ are fixed, then Proposition ~\ref{prop:prop8MO} says that
\begin{equation}
\mathbb{P}([2\ell -1, 2n -2u-1] \subseteq A'+A' \  | \  A'\cap [0, \ell-1] = L, A'\cap [n-u, n-1] = U) > 1 - \frac{1+q}{p^2}(q^{|L|}+ q^{|U|}),
\end{equation}
independent of $n$.
Therefore,
\begin{eqnarray}
&&\mathbb{P}([2\ell -1, 2n -2u-1] \subseteq A'+A' \mbox{ and } A'\cap [0, \ell-1] = L, A'\cap [n-u, n-1] = U) \nonumber\\
&& \ > \ \left(1 - \frac{1+q}{p^2}(q^{|L|}+ q^{|U|})\right)q^{\ell}q^{u}.
\end{eqnarray}
By letting $L = [0, \ell-1], U = [n-u, n-1]$ so the ends fill in, we get that
\begin{equation}
\mathbb{P}(A'+A' = [0, 2n-2]) \  > \ \left(1 - \frac{1+q}{p^2}(q^{\ell}+ q^{u})\right)q^{\ell}q^{u}.
\end{equation}
Pick $\ell,u$ large enough so that the first term in the product is positive, we get that
\begin{equation}
\mathbb{P}(A'+A' = [0, 2n-2]) \  > \ \left(1 - \frac{1+q}{p^2}(q^{s}+ q^{s})\right)q^{s}q^{s}\ =\ \left(1 - \frac{1+q}{p^2}2q^{s}\right)q^{2s},
\end{equation}
which is a constant independent of $n$, as desired.

As $A = k/2 + A'$, we have $A+A = k + A'+A' = [k, 2n-2]$ and so $M_{[0, n-1]}(A) = k$. Thus \begin{align}
\mathbb{P}(M_{[0, n-1]}(A) = k) &\ \ge\ \mathbb{P}(A = k/2 + A' \mbox{ and } M_{n-k/2}(A') = 0) \nonumber\\
& \ =\ q^{k/2} \mathbb{P}(M_{n-k/2}(A') = 0) \nonumber\\
&\ \gg\ q^{k/2}.\label{eqn: lowerboundeven}
\end{align}

This proves the lower bound in Theorem ~\ref{thm:m of k} when $k$ is even.

If $k$ is odd, then we can let $L = [0, \ell-1] \setminus \{ 2,3 \}$ and $U = [n-u, n-1]$ so that only the element $3$ is missing from $A' + A'$. Then we get a bound for $\mathbb{P}(M_{[0, n-1]}(A')=1)$. Letting $A = (k-1)/2 + A'$, we get the desired lower bound in Theorem ~\ref{thm:m of k} for when $k$ is odd.

Now, we find the upper bound. For this, we introduce some notation. We set
\begin{equation}
    M_{[0,n-1]} \ := \ |[0,2n-2]\backslash (A+A)|=2n-1-|A+A|.
\end{equation}
For the upper bound, we have the following inequality for the probability of missing $k$ elements in $[0,n/2]$:
\begin{align}
    \mathbb{P}(|[0,n/2]\backslash (A+A)|=k) & \ \leq\ \mathbb{P}(j\not\in A+A, j\in [k,n/2])\nonumber\\
    &\ \leq \ 2\sum_{j\geq k} (1-p^2)^{j/2}\nonumber\\
    &\ \ll\ (1-p^2)^{k/2}, \label{eq:appendixUpper}
\end{align}
and similarly for $\mathbb{P}(|[3n/2,2n]\backslash (A+A)|=k)$. Furthermore, there is an equation ((7.27) from \cite{LMO}) that connects the probability of missing $k$ elements to the probability of missing elements in $[0,n/2]$ and $[3n/2,2n]$:
\begin{equation}\label{eq:appendixUpper2}
    \mathbb{P}(M_{[0,n-1]}(A)=k) = \sum_{i+j=k}\mathbb{P}(|[0,n/2]\backslash (A+A)|=i)\mathbb{P}(|[3n/2,2n]\backslash (A+A)|=j) + O\pn{(1-p^2)^{n/4}}.
\end{equation}
Combining ~\ref{eq:appendixUpper} and ~\ref{eq:appendixUpper2}, we get
\begin{align}
    &\mathbb{P}(M_{[0,n-1]}(A)=k)\nonumber\\
    &= \ \sum_{i+j=k}\mathbb{P}(|[0,n/2]\backslash (A+A)|=i)\mathbb{P}(|[3n/2,2n]\backslash (A+A)|=j) + O\pn{(1-p^2)^{n/4}}\nonumber\\
    &\ll\ \sum_{i+j=k}(1-p^2)^{i/2}(1-p^2)^{j/2} + (1-p^2)^{n/4}\nonumber\\
    &\ll\ k(1-p^2)^{k/2}+(1-p^2)^{n/4}. \label{eq:appendixUpper3}
\end{align}
Therefore, if $k/2<n/4$, we get
\begin{equation}
    \mathbb{P}(M_{[0,n-1]}(A)=k)\ \ll\ k(1-p^2)^{k/2}.
\end{equation}
However, \cite{LMO} shows we can improve this bound as follows, with the use of \eqref{eqn:decay}:
\begin{align}
    \mathbb{P}(|[0,n/2]\backslash (A+A)|)=k&\ \leq\ \mathbb{P}(A+A \text{ misses 2 elements greater than } k-3)\nonumber\\
    & \ = \ \mathbb{P}(i,j\not\in A+A, i,j\in[k-3,n/2])\nonumber\\
    &\ = \ \sum_{k-3<i<j}\mathbb{P}(i,j\not\in A+A)\nonumber\\
    &\ \ll\ \sum_{k-3<i<j} \bigg(\frac{g(p)+1+p}{2g(p)}\bigg)^{\frac{j-i}{2}}\bigg(\frac{1-p+g(p)}{2}\bigg)^{j+1}\nonumber\\
    &\ \ll\ \bigg(\frac{g(p)+1+p}{2g(p)}\bigg)^{\frac{k-k}{2}}\bigg(\frac{1-p+g(p)}{2}\bigg)^{k+1}\nonumber\\
    &\ = \ \bigg(\frac{1-p+g(p)}{2}\bigg)^{k+1}\ <\ \bigg(\frac{1-p+g(p)}{2}\bigg)^{k}.\label{eq:appendixfinal}
\end{align}
Note that as in \eqref{eq:appendixUpper3}, we always have an extra $(1-p^2)^{n/4}$ term. To make this term negligible, we need to have $(1-p^2)^{n/4} < ((1-p+g(p))/2)^{k}$, which means $n > k \cdot 4 \log((1-p+g(p))/2)/ \log(1-p^2)$. This condition is sufficient in this case where we have the bound $((1-p+g(p))/2)^{k}$. However, in general, we know that we have a lower bound of $(1-p)^{k/2}$ for the distribution. Therefore, to make the $(1-p^2)^{n/4}$ term always negligible, we can have $(1-p^2)^{n/4} < (1-p)^{k/2}$, which means $n > k \cdot 2 \log(1-p)/ \log(1-p^2)$, as in the statement of Theorem ~\ref{thm:m of k}. Note that then the implied constants are independent of n. Combining \eqref{eqn: lowerboundeven} and \eqref{eq:appendixfinal}, we get Theorem ~\ref{thm:m of k}.
\end{proof}

\section{Our Bounds for $\mathbb{P}(|B|=k)$ Are Good}\label{whysharpbound}  

To observe numerically how good our bounds are, we must compare our bounds to the true values of $\mathbb{P}(|B|=k).$ However, $\mathbb{P}(|B|=k)$ cannot be computed directly; thus, we run simulations to estimate $\mathbb{P}(|B|=k)$. We pick $p\in (0,1)$ and run $10^6$ simulations to form subsets of $\{0,1,\dots,400\}$ and find the frequency of each number of missing sums within these $10^6$ simulations. We then compare the plot of the simulated distribution with our bound functions mentioned in Corollary ~\ref{lowboundf} and Corollary ~\ref{upboundf}.

\begin{figure}[ht]
\centering
\includegraphics[scale=.97]{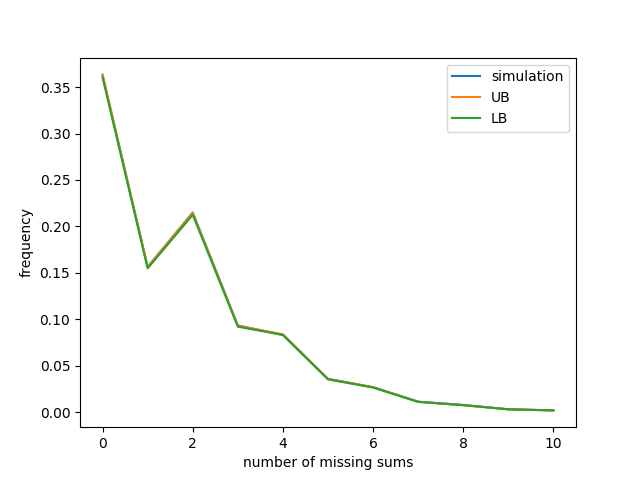}
\caption{For $p=0.8$, lower bound, upper bound and simulation of $\mathbb{P}(|B|=k).$ At $p=0.8$, the lower bound and upper bound for $\mathbb{P}(|B|=k)$ are so good that we cannot differentiate the lines. The two bounds and the simulation seem to closely coincide at all points.}
\label{good8}
\end{figure}



\begin{figure}[ht]
\centering
\includegraphics[scale=.97]{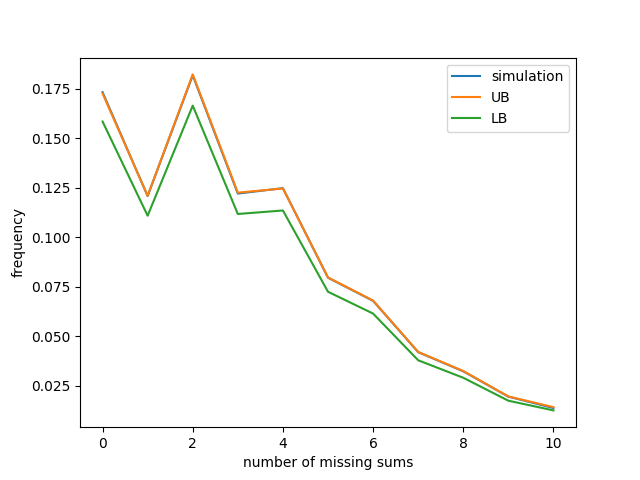}
\caption{At $p=0.7$, the lower bound and upper bound for $\mathbb{P}(|B|=k)$ are still close to each other. The upper bound seems to coincide with the simulation everywhere. However, the bounds are relatively worse compared to the case $p=0.8$.}
\label{good7}
\end{figure}


\begin{figure}[ht]
\centering
\includegraphics[scale=.97]{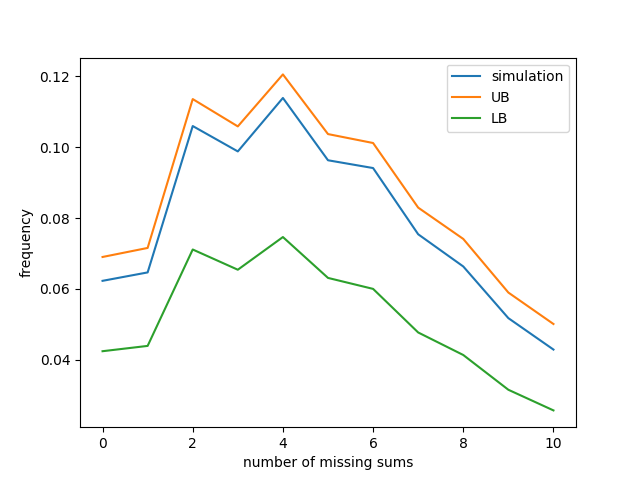}
\caption{At $p=0.6$, the upper bound is fairly good while the lower bound is much worse compared to previous cases.}
\label{good6}
\end{figure}



\newpage
\section{Data for Divot Computations}\label{Data}
All the data we provide below corresponds to $\ell=30$ and $a=12.$ Our program computes all the quantities required by Inequalities \eqref{lowboundf} and \eqref{upboundf} to find lower and upper bounds for $m_p(k)$ when $p$ varies. Our method of storing and collecting data are mentioned at the end of Section \ref{sec:divot}.

\subsubsection{Data for $\min L^{12}_{i}$ and $\tau(L^{12}_i)$}
\begin{center}
\begin{tabular}{ |c|c|c|c|c|c|c|}
\hline
$i$ & $0$ & $1$ & $2$ & $3$ & $4$ & $5$ \\
\hline
$\min L^{12}_{i}$ & 12 & 11 & 11 & 11 & 11 & 11  \\
\hline
$\tau(L^{12}_i)$ & 7 & 7 & 6 & 6 & 6 & 6  \\
\hline

\end{tabular}
\end{center}

\subsubsection{Data for $c_k(i)$ \mbox{ for } $0\le k\le 2$}

\begin{center}
\begin{tabular}{cc}
\begin{minipage}{0.5\textwidth}
\begin{tabular}{|l|l|l|l|}
\hline
$i$ & $c_0(i)$ & $c_1(i)$ & $c_2(i)$ \\ \hline
0-7 & 0 & 0 & 0 \\ \hline
8 & 0 & 0 & 12 \\ \hline
9 & 58 & 1552 & 13955 \\ \hline
10 & 10629 & 82696 & 276434 \\ \hline
11 & 190349 & 704139 & 1495762 \\ \hline
12 & 1164105 & 2613360 & 4544680 \\ \hline
13 & 3879603 & 6121208 & 9753610 \\ \hline
14 & 8720201 & 10586952 & 16142608 \\ \hline
15 & 14730206 & 14526747 & 21525160 \\ \hline
16 & 19817016 & 16371555 & 23716940 \\ \hline
17 & 21916190 & 15421977 & 21913801 \\ \hline
18 & 20269375 & 12251022 & 17114758 \\ \hline
\end{tabular}
\end{minipage}
&
\begin{minipage}{0.5\textwidth}
\begin{tabular}{|l|l|l|l|}
\hline
$i$ & $c_0(i)$ & $c_1(i)$ & $c_2(i)$ \\ \hline
19 & 15817037 & 8237988 & 11333625 \\ \hline
20 & 10452359 & 4689056 & 6359012 \\ \hline
21 & 5847957 & 2251741 & 3010077 \\ \hline
22 & 2759881 & 906081 & 1192562 \\ \hline
23 & 1090747 & 302191 & 390638 \\ \hline
24 & 356894 & 82172 & 103915 \\ \hline
25 & 95055 & 17777 & 21870 \\ \hline
26 & 20099 & 2947 & 3501 \\ \hline
27 & 3248 & 352 & 400 \\ \hline
28 & 377 & 27 & 29\\ \hline
29 & 28 & 1 & 1 \\ \hline
30 & 1 & 0 & 0 \\ \hline
\end{tabular}
\end{minipage}
\end{tabular}
\end{center}
\mbox{}\newline

\subsubsection{Data for $c_{k,a}(i)$}

\begin{center}
\begin{tabular}{|l|l|l|l|}
\hline
$i$ & $c_{0,12}(i)$ & $c_{1,12}(i)$ & $c_{2,12}(i)$ \\ \hline
0-10 & 0 & 0 & 0 \\ \hline
11 & 0 & 5 & 27 \\ \hline
12 & 400 & 3352 & 18890 \\ \hline
13 & 39072 & 198265 & 589832 \\ \hline
14 & 685029 & 1857746 & 3772428 \\ \hline
15 & 3664341 & 6033358 & 9993760 \\ \hline
16 & 9311984 & 10449491 & 15740073 \\ \hline
17 & 14592372 & 12237242 & 17647078 \\ \hline
18 & 16358625 & 10909486 & 15345556 \\ \hline
19 & 14202656 & 7803902 & 10775362 \\ \hline
20 & 9943771 & 4584649 & 6229436 \\ \hline
21 & 5729373 & 2234092 & 2989242 \\ \hline
22 & 2740544 & 904203 & 1190494 \\ \hline
23 & 1088774 & 302096 & 390543 \\ \hline
24 & 356799 & 82172 & 103915 \\ \hline
25 & 95055 & 17777 & 21870 \\ \hline
26 & 20099 & 2947 & 3501 \\ \hline
27 & 3248 & 352 & 400 \\ \hline
28 & 377 & 27 & 29 \\ \hline
29 & 28 & 1 & 1 \\ \hline
30 & 1 & 0 & 0 \\ \hline
\end{tabular}
\end{center}


\ \\


\begin{thebibliography}{DKMMW}

\bibitem[AMMS]{AMMS}
M. Asada, S. Manski, S. J. Miller, and H. Suh, \emph{Fringe pairs in generalized MSTD sets}, International Journal of Number Theory \textbf{13} (2017), no. 10, 2653–2675.

\bibitem[BELM]{BELM}
A. Bower, R. Evans, V. Luo and S. J. Miller, \emph{Coordinate sum and difference sets of $d$-dimensional modular hyperbolas}, INTEGERS \#A31, 2013, 16 pages.

\bibitem[CLMS]{CLMS}
H. Chu, N. Luntzlara, S. J. Miller and L. Shao, \emph{Generalizations of a Curious Family of MSTD Sets Hidden By Interior Blocks}, to appear in Integers.

\bibitem[CMMXZ]{CMMXZ}
H. Chu, N. McNew, S. J. Miller, V. Xu and S. Zhang, \emph{When Sets Can and Cannot Have MSTD Subsets}, Journal of Integer Sequences \textbf{21} (2018), Article 18.8.2. 

\bibitem[DKMMW]{DKMMW}
T. Do, A. Kulkarni, S.J. Miller, D. Moon, and J. Wellens, \emph{Sums and Differences of Correlated Random Sets}, Journal of Number Theory \textbf{147} (2015), 44--68.

\bibitem[H-AMP]{H-AMP}
S. Harvey-Arnold, S. J. Miller and F. Peng, \emph{Distribution of missing differences in diffsets}, preprint.

\bibitem[He]{He}
P. V. Hegarty, \emph{Some explicit constructions of sets with more sums than differences} (2007), Acta Arithmetica \textbf{130} (2007), no. 1, 61--77.

\bibitem[HM]{HM}
P. V. Hegarty and S. J. Miller, \emph{When almost all sets are difference dominated}, Random Structures and Algorithms \textbf{35} (2009), no. 1, 118--136.

\bibitem[HLM]{HLM}
A. Hemmady, A. Lott and S. J. Miller, \emph{When almost all sets are difference dominated in $\mathbb{Z}/n\mathbb{Z}$},  Integers \textbf{17} (2017), Paper No. A54, 15 pp.

\bibitem[ILMZ]{ILMZ}
G. Iyer, O. Lazarev, S. J. Miller and L. Zhang, \emph{Generalized more sums than differences sets,} Journal of Number Theory \textbf{132} (2012), no. 5, 1054--1073.

\bibitem[LMO]{LMO}
O. Lazarev, S. J. Miller, K. O'Bryant, \emph{Distribution of Missing Sums in Sumsets} (2013), Experimental Mathematics \textbf{22}, no. 2, 132--156.

\bibitem[Ma]{Ma}
J. Marica, \emph{On a conjecture of Conway}, Canad. Math. Bull. \textbf{12} (1969), 233--234.

\bibitem[MO]{MO}
G. Martin and K. O'Bryant, \emph{Many sets have more sums than differences}, in Additive Combinatorics, CRM Proc. Lecture Notes, vol. 43, Amer. Math. Soc., Providence, RI, 2007, pp. 287--305.

\bibitem[MOS]{MOS}
S. J. Miller, B. Orosz and D. Scheinerman, \emph{Explicit constructions of infinite families of MSTD sets}, Journal of Number Theory \textbf{130} (2010) 1221--1233.

\bibitem[MS]{MS}
S. J. Miller and D. Scheinerman, \emph{Explicit constructions of infinite families of mstd sets,} Additive Number Theory, Springer, 2010, pp. 229-248.

\bibitem[MPR]{MPR}
S. J.  Miller, S. Pegado and L. Robinson, \emph{Explicit Constructions of Large Families of Generalized More Sums Than Differences Sets}, Integers \textbf{12} (2012), \#A30.

\bibitem[MV]{MV}
S. J. Miler and K. Vissuet, \emph{Most Subsets are Balanced in Finite Groups}, Combinatorial and Additive Number Theory, CANT 2011 and 2012 (Melvyn B. Nathanson, editor), Springer Proceedings in Mathematics \& Statistics (2014), 147--157.

\bibitem[Na1]{Na1}
M. B. Nathanson, \emph{Problems in additive number theory, 1},
Additive combinatorics, 263--270, CRM Proc. Lecture Notes \textbf{43},
Amer. Math. Soc., Providence, RI, 2007.

\bibitem[Na2]{Na2}
M. B. Nathanson, \emph{Sets with more sums than differences},
Integers : Electronic Journal of Combinatorial Number Theory \textbf{7} (2007), Paper A5 (24pp).

\bibitem[PW]{PW}
D. Penman and M. Wells, On sets with more restricted sums than
differences, \textit{Integers} \textbf{13} (2013), \#A57.

\bibitem[Ru1]{Ru1}
I. Z. Ruzsa, \emph{On the cardinality of $A + A$ and $A - A$}, Combinatorics year (Keszthely, 1976), vol. 18, Coll. Math. Soc. J. Bolyai, North-Holland-Bolyai T$\grave{{\rm a}}$rsulat, 1978, 933--938.

\bibitem[Ru2]{Ru2}
I. Z. Ruzsa, \emph{Sets of sums and differences}. In: S\'eminaire de Th\'eorie des Nombres de
Paris 1982-1983, pp. 267--273. Birkh\"auser, Boston (1984).

\bibitem[Ru3]{Ru3}
I. Z. Ruzsa, \emph{On the number of sums and differences}, Acta Math. Sci. Hungar. \textbf{59} (1992), 439--447.

\bibitem[Sp]{Sp}
W. G. Spohn, On Conway's conjecture for integer sets, \textit{Canad.
Math. Bull} \textbf{14} (1971), 461-462.

\bibitem[Zh1]{Zh1}
Y. Zhao, \emph{Constructing MSTD sets using bidirectional ballot sequences}, Journal of Number Theory \textbf{130} (2010), no. 5, 1212--1220.

\bibitem[Zh2]{Zh2}
Y. Zhao, \emph{Sets characterized by missing sums and differences}, Journal of Number Theory \textbf{131} (2011), no. 11, 2107--2134.

\end{thebibliography}
\end{document}